\documentclass[11pt,reqno]{amsart}

\usepackage[utf8]{inputenc}
\usepackage[T1]{fontenc}
\usepackage{amsmath,amssymb,amsthm,amscd,amsfonts}
\usepackage{mathrsfs,enumerate,graphicx,float}
\usepackage[all]{xy}
\usepackage[colorlinks=true,linkcolor=blue,citecolor=red,urlcolor=magenta,pdfusetitle]{hyperref}
\usepackage{tikz-cd}
\usetikzlibrary{calc}
\theoremstyle{plain}
\newtheorem{theorem}{Theorem}[section]
\newtheorem{lemma}[theorem]{Lemma}
\newtheorem{proposition}[theorem]{Proposition}
\newtheorem{corollary}[theorem]{Corollary}

\theoremstyle{definition}
\newtheorem{definition}[theorem]{Definition}
\newtheorem{example}[theorem]{Example}

\theoremstyle{remark}
\newtheorem{remark}[theorem]{Remark}

\numberwithin{equation}{section}

\newcommand{\cA}{\mathcal{A}}
\newcommand{\cB}{\mathcal{B}}
\newcommand{\cC}{\mathcal{C}}

\newcommand{\Proj}{\operatorname{Proj}}

\newcommand{\pd}{\operatorname{pd}}
\newcommand{\gldim}{\operatorname{gl.dim}}
\newcommand{\id}{\operatorname{id}}
\newcommand{\Hom}{\operatorname{Hom}}

\newcommand{\Ext}{\operatorname{Ext}}

\newcommand{\Ker}{\operatorname{Ker}}
\newcommand{\Coker}{\operatorname{Coker}}

\newcommand{\modcat}{\operatorname{mod}}

\newcommand{\dell}{\operatorname{dell}}
\newcommand{\ddell}{\operatorname{ddell}}
\newcommand{\Findim}{\operatorname{Findim}}

\newcommand{\op}{\mathrm{op}}
\newcommand{\into}{\lhook\joinrel\longrightarrow}
\newcommand{\onto}{\twoheadrightarrow}

\begin{document}

\title[Reduction techniques for the derived delooping levels]
{Reduction techniques for the derived delooping levels}

\author{Kaili Wu}
\address{College of Science, Nanjing Forestry University,
Nanjing 210037, Jiangsu, P.~R.~China}
\email{kailywu@163.com}

\author{Jiaqun Wei}
\address{School of Mathematical Sciences,
Zhejiang Normal University,
Jinhua 321004, Zhejiang, P.~R.~China}
\email{weijiaqun5479@zjnu.edu.cn}

\author{Dajun Liu}
\address{School of Mathematics-Physics and Finance,
Anhui Polytechnic University,
Wuhu 241000, Anhui, P.~R.~China}
\email{liudajun@ahpu.edu.cn}

\author{Weiqing Cao}
\address{School of Mathematics and Statistics,
Jiangsu Normal University,
Xuzhou 221116, Jiangsu, P.~R.~China}
\email{weiqingcao@jsnu.edu.cn}

\maketitle

\noindent {\bf Abstract}: The derived delooping level is a recently introduced homological invariant that provides an upper bound for the finitistic dimension of the opposite algebra. In this paper, we employ two reduction techniques—cleft extensions and recollements—to study the finiteness of the derived delooping level of finite-dimensional algebras over a field. By applying the theory of cleft extensions to bound quiver algebras, we establish arrow-removal operations that preserve the finiteness of the derived delooping level. In parallel, using recollement techniques, we develop vertex-removal operations with the same finiteness-preserving property. We conclude with several examples illustrating the applicability and effectiveness of these reduction methods.\vskip 12pt



\section{Introduction}\label{sec:introduction}

The finitistic dimension conjecture asks whether the projective
dimensions of finitely generated modules of finite projective dimension
over an Artin algebra admit a uniform bound. It has been studied using
Igusa--Todorov functions, derived equivalences, idempotent ideals, and
various homological reduction techniques; see, for example,
\cite{IT,PX,JW,XI,XII,XIII}.

A syzygy-theoretic approach was initiated by G\'elinas, who introduced
the delooping level and proved that it gives an upper bound for the big
finitistic dimension of the opposite algebra \cite{Gelinas2022}.
Delooping levels have since been investigated for several classes of
algebras, including radical-square-zero and Nakayama algebras
\cite{G1,R,S}. The relation between the delooping level and the
finitistic dimension was studied further in \cite{BLM}. The limitations
of the delooping level became particularly apparent when Kershaw and
Rickard constructed a finite-dimensional algebra with infinite
delooping level \cite{KershawRickard2024}.

Motivated by the need for a finer invariant, Guo and Igusa introduced
the effective, sub-derived, and derived delooping levels
\cite{GuoIgusa2025}. For a finite-dimensional algebra \(\Lambda\), they
proved that
\[
\Findim\Lambda^{\op}
=
\operatorname{edell}\Lambda
\leq
\ddell\Lambda
\leq
\dell\Lambda.
\]
They also showed that the class of finitely generated modules of finite
derived delooping level is closed under submodules and extensions. For
the Kershaw--Rickard algebra, their calculation gives
\[
\ddell\Lambda=1<\infty=\dell\Lambda.
\]
Thus the derived delooping level retains finitistic-dimension
information that is not detected by the original delooping level. The
symmetry properties of this new invariant were subsequently studied in
\cite{Guo2025}. Related question concerning delooping-type invariants
under stable equivalences is considered in
\cite{XiZhang2026}. These developments motivate the study of
reduction techniques for derived delooping levels undertaken in the
present paper.

 Cleft extensions of abelian categories, introduced by Beligiannis \cite{Beligiannis2000}, simultaneously cover  trivial extensions  and  tensor rings. In particular,  cleft extensions of module categories are precisely $\theta$-extensions \cite[Proposition 6.9]{KP2025}. Several families of rings arise from cleft extensions and their  inclusion relations are summarized below\cite{K2026}:
 \[
\begin{aligned}
\{\text{One-point extensions}\}\ &\ \subseteq \ &\ 
\{\text{Triangular matrix rings}\}\\
\rotatebox{-90}{$\subseteq$}\ \ \ \ \ \ \ & &\rotatebox{-90}{$\subseteq$}\ \ \ \ \ \ \ \ \  \ \ \ \  \ \ \\
\{\text{0-0 Morita context rings}\}\ 
& 
 &
\{\text{Tensor rings}\}\ \ \ \ \ \
\\
\rotatebox{-90}{$\subseteq$}\ \ \ \ \ \ \ & &\rotatebox{-90}{$\subseteq$}\ \ \ \ \ \ \ \ \  \ \ \ \  \ \ \\
\{\text{Trivial extension rings}\}\ 
&\ \subseteq\ 
 &\ 
\{\text{Positively graded rings}\}
&\ \subseteq\ 
 &\ 
\{\theta\text{-extensions}\}.
\end{aligned}
\]
Hence cleft extensions of abelian categories provide the proper conceptual framework  for the study of many problems in several different contexts, such as, finitistic dimensions and global dimensions \cite{EGPS2025,Giatagantzidis2025, GPS2021}, Igusa--Todorov distances, extension dimensions and Rouquier dimensions \cite{MaZhengLiu2026}, and Gorenstein weak (flat-cotorsion) global dimension \cite{LiangMaYang2025}. Motivated by these existing works,  our first goal is to incorporate    derived delooping levels into the  framework of cleft extension for further exploration. 

While our primary motivation comes from results on cleft ring extensions, the arguments in our proofs rely only on the homological machinery of cleft extensions of abelian categories. Moreover, many examples of cleft extensions do not arise from ring-theoretic constructions. It is therefore natural to develop the theory first in the setting of abelian categories and then specialize the resulting statements to module categories.

To this end, we extend the notion of derived delooping levels from module categories to abelian categories and establish some of their basic properties in this more general setting. We then obtain the following quantitative results for cleft extensions.

\begin{theorem}\label{thm:intro-cleft}{\rm(Theorem~\ref{thm:cleft-transfer})}
Let $(\cB,\cA,i,e,l)$ be a cleft extension of length categories with finitely many isomorphism classes of simple objects and enough projective objects. Assume that $l$ is exact and that $e$ preserves projective objects. Set
\[
L_e
=
\max\bigl\{
\ell_{\cB}(eS)
\mid
S\in\operatorname{simp}\cA
\bigr\},
\]
where $\ell_{\cB}(eS)$ denotes the composition length of an object
$eS$ in $\cB$. 
\begin{enumerate}[(a)]
\item 
Suppose that $e$ admits an exact right adjoint $r$ and  $c_r:=\sup\{\pd_{\cA}r(P)\mid P\in\Proj\cB\}<\infty$. Then
\[
\ddell(\cB)
\leq
\ddell(\cA)
\leq
\pd(r)+L_e\bigl(\ddell(\cB)+1\bigr).
\]
Moreover, if $F^n=0$ for some $n\geq1$ and $\pd(r)<\infty$, then, for every $k\geq1$,
\[
k\text{-}\ddell(\cA)<\infty
\quad\Longleftrightarrow\quad
k\text{-}\ddell(\cB)<\infty.
\]

\item Suppose that
$
\operatorname{Im}G\subseteq\operatorname{Proj}\cA,
$ Then, for every $k\geq1$,
\[
k\text{-}\ddell(\cB)
\leq
k\text{-}\ddell(\cA)
\leq
\max\bigl\{
1,\,
L_e\bigl(k\text{-}\ddell(\cB)+1\bigr)-1
\bigr\}.
\]
In particular, $k\text{-}\ddell(\cA)<\infty
\quad\Longleftrightarrow\quad
k\text{-}\ddell(\cB)<\infty.$
\end{enumerate}
\end{theorem}
Here and below, an expression on the right-hand side is understood to
be $\infty$ whenever one of the derived delooping levels occurring in
it is infinite.

As a direct application of part~($b$), we obtain a sharper numerical comparison for the arrow-removal operation studied in
\cite{EPS2022,GPS2021}.

\begin{theorem}[Theorem~\ref{thm:arrow-removal}, Derived-delooping arrow removal]
\label{thm:intro-arrow-removal}
Let $Q$ be a finite quiver and let
$
\Lambda=\Bbbk Q/I
$
be an admissible quotient of the path algebra $\Bbbk Q$ over a field $\Bbbk$. Suppose that
\[
a_i:v_{e_i}\longrightarrow v_{f_i},
\qquad
i=1,\ldots,t,
\]
are arrows of $Q$ which do not occur in a set of minimal generators of $I$ and satisfy
\[
\Hom_{\Lambda}(e_i\Lambda,f_j\Lambda)=0
\qquad
\text{for all }1\leq i,j\leq t.
\]
Set
$
\Gamma
=
\Lambda/
\Lambda\{\overline a_1,\ldots,\overline a_t\}\Lambda.
$
Then, for every $k\geq1$,
\[
k\text{-}\ddell\Gamma
\leq
k\text{-}\ddell\Lambda
\leq
\max\bigl\{
1,\,
k\text{-}\ddell\Gamma
\bigr\}.
\]
\end{theorem}

Recollements, introduced by Beilinson, Bernstein, and Deligne
\cite{BBD}, provide a useful framework for decomposing a category into
two simpler parts and comparing their homological properties. Several
equivalent formulations of recollements of abelian categories have
been developed; see, for example, \cite{FP,PH}. In representation
theory, recollement techniques have been used to study global and
finitistic dimensions and to reduce homological conjectures
\cite{CX,K,QH}. Their connections with idempotent ideals and related
homological invariants have also been investigated
\cite{APT,GLP,XIII}.

Motivated by these developments, we study derived delooping levels in
recollements of abelian categories.  The main results of this part  gives
general criteria for transferring the finiteness of the
$k$-derived delooping level among the three categories of a
recollement. These criteria are formulated in terms of the exactness
of the recollement functors and uniform bounds on the projective
dimensions of the images of projective objects.

\begin{theorem}\label{1}(Theorem \ref{T1})
Let $(\cA,\cB,\cC)$ be a recollement of length categories with finitely many   isomorphism classes of  simple objects and enough projective objects as the following diagram
\[
\begin{tikzcd}[column sep=5.5em]
\cA
\arrow[bend left=0]{r}{i}
&
\cB
\arrow[bend left=0]{r}{e}
\arrow[bend left=40]{l}{p}
\arrow[bend right=40]{l}[above]{q}
&
\cC.
\arrow[bend left=40]{l}{r}
\arrow[bend right=40]{l}[above]{l}
\end{tikzcd}
\]
(a) Suppose that $c_e:=\sup\{\pd_{\cC}e(P)\mid P\in\Proj\cB\}<\infty$,   If  $k\text{-}\mathrm{ddell}\cB<\infty$, then   $k\text{-}\mathrm{ddell} \cC<\infty$.\\
(b) Suppose that $q$ is exact.  If  $k\text{-}\mathrm{ddell}\cB<\infty$, then   $k\text{-}\mathrm{ddell} \cA<\infty$.\\
(c) Suppose that $r$ is exact,  $c_i:=\sup\{\pd_{\cB}i(P)\mid P\in\Proj\cA\}<\infty$ and   $c_r:=\sup\{\pd_{\cB}r(P)\mid P\in\Proj\cC\}<\infty$. If  $k\text{-}\mathrm{ddell}\cA<\infty$ and  $k\text{-}\mathrm{ddell} \cC<\infty$, then  $k\text{-}\mathrm{ddell}\cB<\infty$.
\end{theorem}

A second result in the recollement setting shows that, when the objects
coming from $\cA$ have projective dimension at most one in $\cB$, the
finiteness problem can be reduced directly from the middle category
$\cB$ to the quotient category $\cC$.

\begin{theorem}[Theorem~\ref{thm:2}]\label{2}
Let $(\cA,\cB,\cC)$ be a recollement of length categories with finitely
many   isomorphism classes of  simple objects and enough projective objects. Assume that the
$\cB$-relative projective global dimension of $\cA$, defined by $\operatorname{pgl.dim}_{\cB}\cA
=
\sup\{
\pd_{\cB}i(A)
\mid A\in\cA
\},$ 
is at most one. Then, for every $k\geq1$,
\[
k\text{-}\ddell(\cB)<\infty
\quad\Longleftrightarrow\quad
k\text{-}\ddell(\cC)<\infty.
\]
\end{theorem}
We next apply these categorical results to module categories of
finite-dimensional algebras. Psaroudakis and Vit\'oria \cite{PV}
showed that, for a semiprimary ring, every recollement of module
categories is equivalent to one induced by an idempotent. Since every
finite-dimensional algebra over a field is semiprimary, the
recollements relevant to our applications arise from idempotents. More precisely, let $\Lambda$ be a finite-dimensional bound quiver
algebra and let $e$ be a sum of vertex idempotents. The corner algebra
$e\Lambda e$ is supported on the vertices occurring in $e$. Thus,
passing from $\Lambda$ to $e\Lambda e$ removes the vertices
corresponding to $1-e$; we refer to this operation as
\emph{vertex removal}.  Applying Theorems~\ref{1} and~\ref{2} to
idempotent recollements yields the vertex-removal results in
Corollary~\ref{cor:idempotent-recollement}. The quantitative
recollement estimates further specialize to the triangular-reduction
results in Corollary~\ref{C1}. In particular,   for a lower triangular matrix algebra $\Lambda=
\begin{pmatrix}
A&0\\
{}_BM_A&B
\end{pmatrix},$ 
if $\pd_A M<\infty$, then Corollary~\ref{C1}(c) gives
\[
k\text{-}\ddell\Lambda
\leq
\max\left\{
k\text{-}\ddell A,\,
\pd_A M+1+k\text{-}\ddell B
\right\}.
\]
Finally, we apply these results to  homological-heart reduction.

The paper is organized as follows. In Section~2, we recall the basic
definitions and notation concerning abelian categories, syzygy objects,
projective dimensions, and stable categories. In Section~3, we introduce
delooping levels and derived delooping levels in abelian categories with
enough projective objects. We establish their basic properties,
including invariance under projective summands, syzygy-shifting
properties, and estimates associated with short exact sequences.

Section~4 is devoted to cleft extensions. We first establish stable
syzygy comparison results and transfer inequalities for exact functors
of finite projective dimension. We then study the behavior of derived
delooping levels under the functors associated with a cleft extension
and prove Theorem~\ref{thm:intro-cleft}. Applying this result to module
categories and bound quiver algebras, we obtain the arrow-removal
theorem stated as Theorem~\ref{thm:intro-arrow-removal}. We conclude the
section with an example illustrating how arrow removal replaces the
direct calculation of stable syzygies by a simpler computation over a
quotient algebra.

In Section~5, we study derived delooping levels in recollements of
abelian categories  and prove Theorems~\ref{1} and~\ref{2}.  
These categorical results are subsequently applied to recollements of
module categories induced by idempotents, yielding vertex-removal
criteria for finite-dimensional algebras. We also derive triangular
reduction for lower triangular matrix algebras and use it to obtain a
reduction from a bound quiver algebra to the algebra supported on its
homological heart.  Finally, we  illustrate these reduction methods
with several examples.

\section{Preliminaries}\label{sec:preliminaries}

Throughout the paper, all abelian categories are assumed to be
essentially small and to have enough projective objects unless stated
otherwise. We use the standard terminology for Artin algebras and
stable module categories from \cite{ARS}.

Let $\cA$ be an abelian category. We denote by $\Proj\cA$ the full
subcategory of projective objects and by $\underline{\cA}
=
\cA/[\Proj\cA]$
the stable category obtained by factoring out the ideal of morphisms
that factor through projective objects.

Recall that an abelian category $\cA$ is a \emph{length category}
\cite{Gabriel1973} if it is skeletally small and every object has a
finite composition series. Every object in a length category admits a
finite decomposition into indecomposable objects with local
endomorphism rings, and this decomposition is unique up to isomorphism
and permutation of the summands. If $\cA$ has only finitely many
isomorphism classes of simple objects, we write $
\mathcal S(\cA)$  for a fixed set of representatives of these classes.

For each object $X\in\cA$, choose an epimorphism $P_X\twoheadrightarrow X$ with $P_X\in\Proj\cA$, and define
\[
\Omega_{\cA}X
=
\Ker(P_X\twoheadrightarrow X).
\]
The chosen syzygies induce an additive endofunctor $\Omega_{\cA}:\underline{\cA}\longrightarrow\underline{\cA},$
which is again denoted by $\Omega_{\cA}$. An object $U$ is called a
\emph{stable retract} of an object $V$ if $U$ is a retract of $V$ in
$\underline{\cA}$; equivalently, there exist morphisms
\[
\underline{s}:U\longrightarrow V,
\qquad
\underline{p}:V\longrightarrow U
\]
in $\underline{\cA}$ such that
$\underline{p}\,\underline{s}
=
\underline{1}_U.$
\begin{definition}[{\cite[Definition~3.1]{MaZhengLiu2026}}]
\label{def:syzygy-object}
Let $X\in\cA$ and let $n\geq1$. An object $K\in\cA$ is called an
\emph{$n$-syzygy object} of $X$ if there is an exact sequence
\[
0\longrightarrow K
\longrightarrow P_{n-1}
\longrightarrow\cdots
\longrightarrow P_0
\longrightarrow X
\longrightarrow0
\]
with $P_i\in\Proj\cA$ for all $0\leq i\leq n-1$. By convention, $X$
itself is a $0$-syzygy object of $X$.
\end{definition}

The distinction between a chosen syzygy and a syzygy object is immaterial in the stable category. This will be proved in Lemma~\ref{lem:schanuel}.

When working with algebras, $\Lambda$ denotes a finite-dimensional
algebra over an algebraically closed field $\Bbbk$, unless stated
otherwise. We write $\operatorname{mod}\text{-}\Lambda$ 
for the category of finitely generated right $\Lambda$-modules and and $\operatorname{proj}(\Lambda)$ for its full subcategory consisting of finitely generated projective right $\Lambda$-modules. For
$X\in\operatorname{mod}\text{-}\Lambda$, we use $\operatorname{pd}_{\Lambda}X,\ 
\operatorname{id}_{\Lambda}X$
for its projective and injective dimensions, respectively, and $
\Omega_{\Lambda}^nX,$ $\Omega_{\Lambda}^{-n}X$
for its $n$th syzygy and $n$th cosyzygy, respectively. We set $
\Omega_{\Lambda}^{0}X=X.$ 
Moreover, $\operatorname{add}_{\Lambda}X$ denotes the full subcategory
of $\operatorname{mod}\text{-}\Lambda$ consisting of all direct
summands of finite direct sums of copies of $X$.

We next extend several standard properties of syzygies from module
categories to abelian categories with enough projective objects.

\begin{lemma}[Schanuel's lemma]\label{lem:schanuel}
Let $X\in\cA$, and let
\[
0\longrightarrow K\longrightarrow P\longrightarrow X
\longrightarrow0
\]
and
\[
0\longrightarrow K'\longrightarrow P'\longrightarrow X
\longrightarrow0
\]
be exact sequences with $P,P'\in\Proj\cA$. Then
\[
K\oplus P'
\cong
K'\oplus P.
\]
More generally, any two $n$-syzygy objects of $X$ become isomorphic
after adding suitable projective direct summands. In particular, they
are isomorphic in $\underline{\cA}$.
\end{lemma}

\begin{proof}
Let $W=P\times_XP'$ be the pullback. Thus there is a commutative square
\[
\begin{CD}
&&&&0&&0\\
&&&&@VVV@VVV\\
&&&&K@=K\\
&&&&@VVV@VVV\\
0@>>>K'@>>>W@>>>P'@>>>0\\
&&@|@VVV@VVV\\
0@>>>K'@>>>P@>>>X@>>>0\\
&&&&@VVV@VVV\\
&&&&0&&0
\end{CD}
\]
Since $P$ and $P'$ are projective, \[
W\simeq K'\oplus P\simeq K\oplus P',
\]
which proves the first assertion. For the higher assertion, compare the last projective deflations in two $n$-step projective resolutions of $X$, apply the first assertion, and continue inductively to the kernels. At the $n$-th step one obtains an isomorphism after adjoining finite direct sums of the projective objects occurring in the two resolutions.
\end{proof}

\begin{lemma}\label{lem:rotation}(syzygy shifting)
Let
\[
0\longrightarrow X\xrightarrow{u}Y\xrightarrow{v}Z\longrightarrow0
\]
be exact. If $P_Z\onto Z$ is a projective deflation with kernel $\Omega Z$, then there is an exact sequence
\[
0\longrightarrow\Omega Z\longrightarrow X\oplus P_Z
\longrightarrow Y\longrightarrow0.
\]
\end{lemma}
\begin{proof}
Form the pullback $W=Y\times_ZP_Z$. It is contained in a pullback diagram
\[
\begin{CD}
0@>>>X@>>>W@>>>P_Z@>>>0\\
@.@|@VVV@VVV\\
0@>>>X@>{u}>>Y@>{v}>>Z@>>>0.
\end{CD}
\]
The upper row splits because $P_Z$ is projective, and therefore $W\simeq X\oplus P_Z$. The other pullback projection fits into an exact sequence
\[
0\longrightarrow\Omega Z\longrightarrow W\longrightarrow Y\longrightarrow0.
\]
After replacing $W$ by $X\oplus P_Z$, this is the required sequence.
\end{proof}
\begin{lemma}\label{lem:iterated-rotation}
Let
\[
0\longrightarrow M_t\longrightarrow M_{t-1}\longrightarrow\cdots
\longrightarrow M_1\longrightarrow M_0\longrightarrow0
\]
be exact and let $s\geq1$. There exist projective objects $Q_1,\ldots,Q_{t-1}$ and an exact sequence
\[
0\longrightarrow\Omega^sM_t\longrightarrow
\Omega^sM_{t-1}\oplus Q_{t-1}\longrightarrow\cdots \Omega^sM_1\oplus Q_{1}
\longrightarrow\Omega^sM_0\longrightarrow0.
\]
\end{lemma}

\begin{proof}
Decompose the given sequence into short exact sequences
\[
0\to K_{j+1}\to M_j\to K_j\to0
\qquad(0\leq j<t),
\]
where $K_0=0$ and $K_t=M_t$. Applying  Horseshoe Lemma to the above short exact sequences, we get the following exact sequences
\[
0\to \Omega^s K_{j+1}\to \Omega^s M_j\oplus Q_{j} \to \Omega^s K_j\to0
\qquad(0\leq j<t).
\]
Splicing the resulting exact sequences proves the assertion.
\end{proof}

\section{Derived delooping levels in the abelian category }\label{sec:abelian-cleft}

In this section,  we give the definition of derived delooping levels in the abelian category and generalize some properties from module categories  to    abelian categories.

\begin{definition}\label{def:abelian-dell-ddell}
{\rm\cite[Definitions~2.5 and 2.22]{GuoIgusa2025}}
Let $X\in\cA$ and $k\geq1$.
\begin{enumerate}[$(a)$]
\item The $k$-delooping level of $X$ is
\[
k\text{-}\dell_{\cA}X
=
\inf\{n\geq0\mid
\Omega_{\cA}^{n}X\text{ is a stable retract of }
\Omega_{\cA}^{n+k}Y
\text{ for some }Y\in\underline{\cA}\}.
\]
\item The $k$-derived delooping level of $X$ is
\[
\begin{split}
k\text{-}\ddell_{\cA}X=\inf\{m\geq0\mid{}&
\text{there exist }0\leq t\leq m\text{ and an exact sequence}\\
&0\to C_t\to C_{t-1}\to\cdots\to C_0\to X\to0,\\
&(i+k)\text{-}\dell_{\cA}C_i\leq m-i
\text{ for }0\leq i\leq t\}.
\end{split}
\]
\end{enumerate}
For $k=1$ we write $\dell_{\cA}X$ and $\ddell_{\cA}X$. 
\end{definition}

\begin{remark}\label{rem:relation-known-definition}
When $\cA=\modcat\text{-}\Lambda$ for a finite-dimensional algebra $\Lambda$, Definition~\ref{def:abelian-dell-ddell} is precisely the definition of Guo--Igusa \cite{GuoIgusa2025}.  The case $t=0$ shows
\[
k\text{-}\ddell_{\cA}X\leq k\text{-}\dell_{\cA}X.
\]
Moreover, every first syzygy has ordinary delooping level zero.
\end{remark}

\begin{definition}\label{def:global-derived-delooping}
Let $\cA$ be a length category with finitely many isomorphism classes of simple objects. For $k\geq1$ set
\[
k\text{-}\ddell(\cA)
=
\sup\{k\text{-}\ddell_{\cA}S\mid S\in\mathcal{S}(\cA)\}.
\]
We write $\ddell(\cA)=1\text{-}\ddell(\cA)$. If $\Lambda$ is a finite-dimensional algebra and $\cA=\modcat \text{-}\Lambda$, we abbreviate these invariants to $k\text{-}\ddell \Lambda$ and $\ddell \Lambda$, respectively.
\end{definition}

\begin{lemma}\label{lem:stable-projective-invariance}
Let $X,Y\in\cA$ and  $k\geq1$.
\begin{enumerate}[$(a)$]
\item If $X$ is a stable retract of $Y$, then
\[
k\text{-}\dell_{\cA}X\leq k\text{-}\dell_{\cA}Y.
\]
\item If $X\oplus P\simeq Y\oplus Q$ for projective objects $P,Q$, then
\[
k\text{-}\dell_{\cA}X=k\text{-}\dell_{\cA}Y
\quad\text{and}\quad
k\text{-}\ddell_{\cA}X=k\text{-}\ddell_{\cA}Y.
\]

\end{enumerate}
\end{lemma}

\begin{proof}
Part~(a) follows directly from the definition of the $k$-delooping
level in the stable category. It also follows that adding a projective
summand does not change the $k$-delooping level. Hence the first
equality in Part~(b) is immediate.

It remains to prove the assertion for the $k$-derived delooping level.
It is enough to show that, for every projective object $P$, $
k\text{-}\ddell_{\cA}(X\oplus P)
=
k\text{-}\ddell_{\cA}X.
$ Put $
m=k\text{-}\ddell_{\cA}X$ 
and suppose first that $m<\infty$. By definition, there is an integer
$t\leq m$ and an exact sequence
\[
0\longrightarrow C_t
\longrightarrow\cdots
\longrightarrow C_0
\xrightarrow{d_0}X
\longrightarrow0
\]
such that $(i+k)\text{-}\dell_{\cA}C_i
\leq m-i$ $(0\leq i\leq t).$ 
Then
\[
0\longrightarrow C_t
\longrightarrow\cdots
\longrightarrow C_0\oplus P
\xrightarrow{d_0\oplus1_P}
X\oplus P
\longrightarrow0
\]
is exact. Since $C_0$ and $C_0\oplus P$ are isomorphic in the stable
category, $k\text{-}\dell_{\cA}(C_0\oplus P)
=
k\text{-}\dell_{\cA}C_0.$
Thus the preceding sequence satisfies the required delooping-level
bounds and gives
\[
k\text{-}\ddell_{\cA}(X\oplus P)
\leq
k\text{-}\ddell_{\cA}X.
\]

Conversely, put $n=k\text{-}\ddell_{\cA}(X\oplus P)$ 
and suppose that $n<\infty$. Choose an integer $s\leq n$ and an exact
sequence
\[
0\longrightarrow D_s
\longrightarrow\cdots
\longrightarrow D_0
\xrightarrow{f}
X\oplus P
\longrightarrow0
\]
such that $(i+k)\text{-}\dell_{\cA}D_i
\leq n-i
\qquad
(0\leq i\leq s).$
Let $D'_0$ be the pullback of $f:D_0\to X\oplus P$ along the canonical
inclusion $X\to X\oplus P$. We obtain a commutative diagram with exact
rows and columns:
\[
\begin{CD}
&&&&0&&0\\
&&&&@VVV@VVV\\
0@>>>\Ker f@>>>D'_0@>>>X@>>>0\\
&&@|@VVV@VVV\\
0@>>>\Ker f@>>>D_0@>>>X\oplus P@>>>0\\
&&&&@VVV@VVV\\
&&&&P@=P\\
&&&&@VVV@VVV\\
&&&&0&&0.
\end{CD}
\]
Since $P$ is projective, the exact sequence
\[
0\longrightarrow D'_0
\longrightarrow D_0
\longrightarrow P
\longrightarrow0
\]
splits. Hence $D_0\cong D'_0\oplus P.$ 
By the already established invariance of $k$-delooping levels under
projective summands, $k\text{-}\dell_{\cA}D'_0
=
k\text{-}\dell_{\cA}D_0.$ 
Splicing the upper row with the preceding part of the original exact
sequence gives
\[
0\longrightarrow D_s
\longrightarrow\cdots
\longrightarrow D_1
\longrightarrow D'_0
\longrightarrow X
\longrightarrow0.
\]
All the required delooping-level bounds remain unchanged. Therefore
\[
k\text{-}\ddell_{\cA}X
\leq n
=
k\text{-}\ddell_{\cA}(X\oplus P).
\]
Combining the two inequalities yields
\[
k\text{-}\ddell_{\cA}(X\oplus P)
=
k\text{-}\ddell_{\cA}X.
\]

Finally, if $X\oplus P\cong Y\oplus Q$ with $P$ and $Q$ projective,
then
\[
\begin{aligned}
k\text{-}\ddell_{\cA}X
&=
k\text{-}\ddell_{\cA}(X\oplus P)\\
&=
k\text{-}\ddell_{\cA}(Y\oplus Q)\\
&=
k\text{-}\ddell_{\cA}Y.
\end{aligned}
\]
This proves Part~(b).
\end{proof}

\begin{lemma}\label{lem:ddell-syzygy}
Let $\mathcal A$ be an abelian category with enough projective
objects. For every $M\in\mathcal A$ and every $k\geq1$,
\[
k\text{-}\ddell_{\mathcal A}(\Omega M)
\leq
k\text{-}\ddell_{\mathcal A}M.
\]
Consequently, $k\text{-}\ddell_{\mathcal A}M<\infty
\Longrightarrow
k\text{-}\ddell_{\mathcal A}(\Omega M)<\infty.$
\end{lemma}

\begin{proof}
Put $m=k\text{-}\ddell_{\mathcal A}M<\infty.$ 
Choose an $m$-witness
\[
0\to C_n\to C_{n-1}\to\cdots\to C_0\to M\to0,
\qquad n\leq m,
\]
such that $(i+k)\text{-}\dell_{\mathcal A}C_i\leq m-i$  $(0\leq i\leq n).$
By Lemma \ref{lem:iterated-rotation}, taking 1-syzygies 
gives an exact sequence
\[
0\to\Omega C_n
\to\Omega C_{n-1}\oplus P_{n-1}
\to\cdots
\to\Omega C_0\oplus P_0
\to\Omega M\to0,
\]
where all $P_i$ are projective.

For every object $X$ and every $q\geq1$, one has $q\text{-}\dell_{\mathcal A}(\Omega X)
\leq q\text{-}\dell_{\mathcal A}X.$ 
Since adding a projective summand does not change the delooping
level, it follows that
\[
\begin{aligned}
(i+k)\text{-}\dell_{\mathcal A}
   (\Omega C_i\oplus P_i)
&=(i+k)\text{-}\dell_{\mathcal A}(\Omega C_i)\\
&\leq(i+k)\text{-}\dell_{\mathcal A}C_i\\
&\leq m-i.
\end{aligned}
\]
Thus the displayed exact sequence is an $m$-witness for
$\Omega M$, and hence
\[
k\text{-}\ddell_{\mathcal A}(\Omega M)\leq m.
\]
\end{proof}

\begin{proposition}\label{prop:extension-estimate}
Let $0\to X\to Y\to Z\to0$ be exact. Then, for every $k\geq1$,
\[
k\text{-}\ddell_{\cA}Y
\leq k\text{-}\ddell_{\cA}X+k\text{-}\ddell_{\cA}Z+1,
\]
and
\[
k\text{-}\ddell_{\cA}X
\leq k\text{-}\ddell_{\cA}Y+k\text{-}\ddell_{\cA}Z+1.
\]
\end{proposition}

\begin{proof}
There is nothing to prove if one of the two terms on the right is infinite. Put
\[
m_1=k\text{-}\ddell X,
\qquad
m_2=k\text{-}\ddell Z.
\]
Choose witnesses of lengths $p\leq m_1$ and $q\leq m_2$:
\begin{equation}\label{eq:X}
0\to D_p\xrightarrow{d_p}\cdots\to D_0\xrightarrow{d_0} X\to0,
\qquad
(i+k)\text{-}\dell D_i\leq m_1-i,
\end{equation}
\begin{equation}\label{eq:Z}
0\to E_q\to\cdots\to E_0\to Z\to0,
\qquad
(j+k)\text{-}\dell E_j\leq m_2-j.
\end{equation}
We also choose the first $p+1$ steps of a projective resolution of $Z$
\[
0 \to \Omega^{p+1}Z \xrightarrow{\pi}P_{p} \xrightarrow{\pi_{p}} \cdots \to P_1 \xrightarrow{\pi_1} P_0 \xrightarrow{\pi_0} Z \to 0,
\]
and another exact sequence by taking($p+1$)-syzygy of  \eqref{eq:Z}
\[
0 \to \Omega^{p+1}E_{q} \to \Omega^{p+1}E_{q-1}\oplus Q_{q-1} \to \cdots \to \Omega^{p+1}E_0\oplus Q_0  \xrightarrow{\partial}  \Omega^{p+1}Z \to 0
\]
where all $Q_i$ are projective objects. Now we prove that there is a  long exact sequence
\begin{equation}\label{eq:splicing}
\begin{split}
0\to&\Omega^{p+1}E_q\to \Omega^{p+1}E_{q-1}\oplus Q_{q-1}\to\cdots
\to\Omega^{p+1}E_0\oplus Q_0\\
& \xrightarrow{f}  D_p\oplus P_p\to\cdots\to D_0\oplus P_0\to Y\to0.
\end{split}
\end{equation}
where $f$ maps into $P_{p}$, which can be factored as $\Omega^{p+1}E_0\oplus Q_{0} \stackrel{\partial} \to \Omega^{p+1}Z \stackrel{\pi} \hookrightarrow P_{p}$. In fact, we may proceed by induction for the proof. Consider the following commutative exact diagram
$$\begin{tikzcd}
& && P_0 \arrow[d, "\pi_0"] \arrow[ld, dashed, "\gamma"]  &\\
0 \arrow[r] & X \arrow[r, "\alpha"] & Y \arrow[r, "\beta"] & Z \arrow[r] & 0,
\end{tikzcd}$$
Since $P_0$ is projective, such a lift $\gamma$ exists. We therefore have
the following commutative exact diagram
$$\begin{tikzcd}
&0\arrow[d]&0\arrow[d]&0\arrow[d]&\\
0\arrow[r]&K_1\arrow[r]\arrow[d]&\text{Ker}( \alpha d_0 \ \gamma)\arrow[r]\arrow[d]&\Omega Z\arrow[r]\arrow[d]&0\\
0\arrow[r]&D_0\arrow[r,"\begin{pmatrix}
   1\\
  0
  \end{pmatrix}"]\arrow[d, "d_0"] &D_0\oplus P_0\arrow[r,"(0\ 1 )"]\arrow[d,"(\alpha d_0 \ \gamma)"] & P_0 \arrow[d, "\pi_0"] \arrow[r]  &0.\\
0 \arrow[r] & X \arrow[r, "\alpha"] \arrow[d]& Y \arrow[r, "\beta"] \arrow[d]& Z \arrow[r]\arrow[d] & 0\\
&0&0&0&
\end{tikzcd}
$$
where the surjectivity of $( \alpha d_0 \ \gamma)$ follows from the Snake
Lemma. Repeating this pullback--kernel construction with the successive
maps $D_i\to D_{i-1}$ and $P_i\to P_{i-1}$ gives, by induction on $i$, the
exact sequence \eqref{eq:splicing}. The terms $D_i\oplus P_i$ occur in
positions $i$, and the term $\Omega^{p+1}E_j\oplus Q_j$ occurs in position
$p+1+j$.
So, for $k$-ddell $Y \leq m_1 + m_2 + 1$, it remains to check
\noindent\begin{enumerate}
    \item $(i+k)$-$\operatorname{dell} (D_i\oplus P_i) \leq m_1 + m_2 + 1 - i$ for $i = 0, \dots, p$,
    \item $(j+ p + 1 + k)$-$\operatorname{dell}  (\Omega^{p+1}E_j\oplus Q_j)\leq m_1 + m_2 - j - p$ for $j= 0, \dots, q$.
\end{enumerate}
Due to the conditions in \eqref{eq:X}, \[
(i+k)\text{-}\dell(D_i\oplus P_i)
\leq m_1-i\leq m_1+m_2+1-i.
\]
For the second statement,  choose $a_j\leq m_2-j$, an object $N_j$, and a stable retraction
\[
\Omega^{a_j}E_j\longrightarrow
\Omega^{a_j+j+k}N_j\longrightarrow
\Omega^{a_j}E_j.
\]
Since $a_j\leq m_2-j$, the integer $s_j=m_1+m_2-j+1-a_j$ 
is nonnegative. Applying $\Omega^{s_j}$ to this retraction gives
\[
\Omega^{m_1+m_2-j+1}E_j
\quad\text{as a stable retract of}\quad
\Omega^{m_1+m_2+k+1}N_j.
\]
Rewriting the two exponents gives
\[
\Omega^{m_1+m_2-p-j}(\Omega^{p+1}E_j)
\quad\text{as a stable retract of}\quad
\Omega^{m_1+m_2-p-j+p+1+j+k}N_j.
\]
The required inequality follows from Lemma~\ref{lem:stable-projective-invariance}. The exact sequence \ref{eq:splicing} therefore witnesses
\[
k\text{-}\ddell Y\leq m_1+m_2+1.
\]

It remains to  prove the second inequality.   By the syzygy-shifting lemma, there exists a projective
object $P_Z$ and an exact sequence
\[
0\longrightarrow\Omega Z
\longrightarrow X\oplus P_Z
\longrightarrow Y
\longrightarrow0.
\]
Applying the first inequality to this exact sequence yields
\[
k\text{-}\ddell_{\cA}(X\oplus P_Z)
\leq
k\text{-}\ddell_{\cA}(\Omega Z)
+k\text{-}\ddell_{\cA}Y+1.
\]
Hence, 
\[
\begin{aligned}
k\text{-}\ddell_{\cA}X
&=
k\text{-}\ddell_{\cA}(X\oplus P_Z)\\
&\leq
k\text{-}\ddell_{\cA}(\Omega Z)
+k\text{-}\ddell_{\cA}Y+1\\
&\leq
k\text{-}\ddell_{\cA}Z
+k\text{-}\ddell_{\cA}Y+1.
\end{aligned}
\]
as desired.
\end{proof}

\begin{lemma}\label{lem:subobject-estimate}
If $X\into Y$, then $\ddell_{\cA}X\leq\ddell_{\cA}Y+1.$
\end{lemma}

\begin{proof}
Let $Z=Y/X$. By Lemma~\ref{lem:rotation}, there is an exact sequence
\[
0\longrightarrow\Omega Z\longrightarrow X\oplus P_Z\longrightarrow Y\longrightarrow0.
\]
Since $\Omega Z$ is a first syzygy, $\dell_{\cA}(\Omega Z)=0$ and hence $ \ddell_{\cA}(\Omega Z)=0.$ 
Proposition~\ref{prop:extension-estimate} and Lemma~\ref{lem:stable-projective-invariance} give
\[
\ddell X
=
\ddell(X\oplus P_Z)
\leq\ddell(\Omega Z)+\ddell Y+1
=\ddell Y+1.
\]
\end{proof}

\begin{remark}\label{rem:recovery-guo-igusa}
For $\cA=\modcat\text{-}\Lambda$, Proposition~\ref{prop:extension-estimate} and  Lemma~\ref{lem:subobject-estimate} recover, respectively, Lemma~3.1 and Theorem~5.2 of Guo--Igusa \cite{GuoIgusa2025}. The new point is that their proofs require only projective deflations, pullbacks, Snake Lemma  and Schanuel comparison; consequently the entire exact calculus is intrinsic to abelian categories with enough projective objects.
\end{remark}

The following observation is useful when the kernel of an epimorphism
is projective.

\begin{lemma}\label{L-projective-kernel-ddell}
Let
\[
0\longrightarrow P\xrightarrow{j}M\xrightarrow{\pi}N
\longrightarrow0
\]
be a short exact sequence, where $P$ is projective. Then, for every
$k\geq1$,
\[
k\text{-}\operatorname{ddell}N
\leq
\max\{1,k\text{-}\operatorname{ddell}M\}.
\]
\end{lemma}

\begin{proof}
Set
$
m=k\text{-}\operatorname{ddell}M<\infty.
$ We first consider the case $m=0$.  Using the given short exact sequence
\[
0\longrightarrow P\longrightarrow M\longrightarrow N
\longrightarrow0,
\]
we have $k\text{-}\operatorname{dell}M=0\leq1$ and $(k+1)\text{-}\operatorname{dell}P=0.$
Hence $k\text{-}\operatorname{ddell}N\leq1.$

Now assume that $m\geq1$. Choose an exact sequence
\[
0\longrightarrow C_n\longrightarrow C_{n-1}
\longrightarrow\cdots\longrightarrow C_1
\xrightarrow{d_1}C_0\xrightarrow{d_0}M
\longrightarrow0
\]
with $n\leq m$ such that
$
(i+k)\text{-}\operatorname{dell}C_i
\leq m-i$ $(0\leq i\leq n).
$

If $n=0$, then $C_0\simeq M$ and $k\text{-}\operatorname{dell}M\leq m.$ The given short exact sequence
\[
0\longrightarrow P\longrightarrow M\longrightarrow N
\longrightarrow0
\]
therefore satisfies $k\text{-}\operatorname{dell}M\leq m$ and $
(k+1)\text{-}\operatorname{dell}P=0\leq m-1.$ Thus $k\text{-}\operatorname{ddell}N\leq m.$

Assume now that $n\geq1$ and set $K=\ker(\pi d_0).$ There is a pullback diagram
\[
\xymatrix{
&0\ar[d]&0\ar[d]&&\\
&\ker p \ar[d]\ar[r]^{\cong}
    & \ker d_0\ar[d]&&\\
0\ar[r]
    &K\ar[d]^q\ar[r]^{\kappa}
    &C_0\ar[d]^{d_0}\ar[r]^{\pi d_0}
    &N\ar@{=}[d]\ar[r]&0\\
0\ar[r]
    &P\ar[r]^j\ar[d]
    &M\ar[r]^{\pi}\ar[d]
    &N\ar[r]&0\\
&0&0&&.
}
\]
Since $P$ is projective, the short exact sequence in the second column splits. Thus there is an isomophism $\iota: K\xrightarrow{\cong}\ker d_0\oplus P.$ 
The exact sequence in the third row, together with the above isomorphism, yields the following exact sequence
\begin{equation}\label{spl}
0\longrightarrow \ker d_0\oplus P \to C_0\to N\to  0.\end{equation}
 Since the original sequence is exact at
$C_0$, the morphism $d_1$ factors as
\[
C_1\xrightarrow{\overline d_1} \ker d_0
\xrightarrow{u}C_0,
\]
where $\overline d_1$ is an epimorphism. Construct a homomorphism 
\[
\alpha=\begin{pmatrix}
\overline d_1&0\\
0& 1_P
\end{pmatrix}:C_1\oplus P\to \ker d_0\oplus P.
\]
Consequently, $\alpha$ is an epimorphism and
$
\ker\alpha\simeq\ker\overline d_1=\ker d_1.
$ So there is an exact sequence 
\begin{equation}\label{spl1}
0\longrightarrow \ker d_1 \to C_1\oplus P\to \ker d_0\oplus P
\to 0.\end{equation}
Combined with exact sequences  \eqref{spl} and \eqref{spl1}, we have 
\[
0\to C_n\to\cdots
\to C_2
\to C_1\oplus P
\to 
C_0
\to N
\to0
\]
is exact.
Since adding a projective direct summand does not change any higher
delooping level,
\[
(k+1)\text{-}\operatorname{dell}(C_1\oplus P)
=
(k+1)\text{-}\operatorname{dell}C_1
\leq m-1.
\]
For all other terms, the original inequalities remain unchanged:
$
(i+k)\text{-}\operatorname{dell}C_i
\leq m-i.
$
Therefore
$
k\text{-}\operatorname{ddell}N\leq m.
$

Combining the above cases yields
\[
k\text{-}\operatorname{ddell}N
\leq
\max\{1,k\text{-}\operatorname{ddell}M\}.
\]
\end{proof}

\section{Derived delooping levels and   cleft extensions}\label{sec:examples}

In this section, we study the behavior of derived delooping levels under
cleft extensions. We then apply the resulting comparison theorem to the
arrow-removal operation and provide examples illustrating its use. We
begin by recalling the notion of a cleft extension.

\begin{definition}\label{def:cleft-extension}
{\rm\cite[Definition~2.1]{Beligiannis2000}}
A \emph{cleft extension} of an abelian category $\cB$ consists of a
diagram of functors
\[
\xymatrix@C=3.7pc{
\cB\ar[r]^-i&
\cA\ar[r]^-e&
\cB\ar@<-1.8ex>[l]_-l
}
\]
such that $e$ is faithful and exact, $(l,e)$ is an adjoint pair, and
there is a natural isomorphism
$
e i \cong \id_{\cB}.
$
We denote such a cleft extension by $(\cB,\cA,i,e,l)$.
\end{definition}

The next result collect several basic properties of a cleft extension.

\begin{lemma}\label{lem:cleft-properties}
{\rm\cite[Lemma~2.2 and sequences~(4.1)--(4.2)]{GPS2021}}
Let $(\cB,\cA,i,e,l)$ be a cleft extension. Then the following statements
hold.
\begin{enumerate}[$(a)$]
\item The functor $e\colon\cA\to\cB$ is essentially surjective.

\item The functor $i\colon\cB\to\cA$ is fully faithful and exact.

\item The functor $l\colon\cB\to\cA$ is faithful and preserves projective
objects.

\item The functor $i$ admits a left adjoint $
q\colon\cA\longrightarrow\cB, $ and there is a natural isomorphism $q l\cong\id_{\cB}.$

\item There exist endofunctors $
F\colon\cB\longrightarrow\cB$ and
$G\colon\cA\longrightarrow\cA$ 
fitting into natural exact sequences of functors
\begin{equation}\label{eq:2.1}
0\longrightarrow F\longrightarrow e l
\longrightarrow\id_{\cB}\longrightarrow0
\end{equation}
and
\begin{equation}\label{eq:2.2}
0\longrightarrow G\longrightarrow l e
\longrightarrow\id_{\cA}\longrightarrow0.
\end{equation}
\end{enumerate}
\end{lemma}

\begin{definition}\label{def:cleft-coextension}
{\rm\cite[Definition~2.1]{Beligiannis2000}}
A cleft coextension is a diagram
\[
\xymatrix@C=3.7pc{
\cB\ar[r]^-i&\cA\ar[r]^-e&\cB\ar@<1.8ex>[l]^-r
}
\]
such that $e$ is faithful and exact, $(e,r)$ is an adjoint pair, and there is a natural isomorphism  $ei\simeq\id_{\cB}$. It is denoted by $(\cB,\cA,i,e,r)$.
\end{definition}
Like with cleft extensions, there  are endofunctors $F'$: $\mathcal{B}\to \mathcal{B}$ and $G': \mathcal{A} \to \mathcal{A}$ that appear in the following exact sequences of functors:
\begin{equation}\label{eq:2.3}
0\longrightarrow \operatorname{id}_{\mathcal{B}}\longrightarrow er\longrightarrow F' \longrightarrow 0, 
\end{equation}
\begin{equation}\label{eq:2.4}
0\longrightarrow \operatorname{id}_{\mathcal{A}}\longrightarrow re\longrightarrow G'\longrightarrow 0.
\end{equation}

\begin{lemma}\label{lem:F-G-nilpotent}
Let the cleft extension $(\cB,\cA,i,e,l)$ be the upper part of the
cleft coextension $(\cB,\cA,i,e,r)$. Then, for every $n\geq 1$,
\[
F^n=0
\quad\Longleftrightarrow\quad
G^n=0
\quad\Longleftrightarrow\quad
(F')^n=0
\quad\Longleftrightarrow\quad
(G')^n=0.
\]
\end{lemma}

\begin{proof}
By \cite[Lemma~2.2]{MaZhengLiu2026} and
\cite[Lemma~2.5]{MaZhengLiu2026}, respectively, we have
\[
F^n=0\quad\Longleftrightarrow\quad G^n=0
\ \text{and}\ 
(F')^n=0\quad\Longleftrightarrow\quad (G')^n=0.
\]
It therefore remains to relate $F^n$ and $(F')^n$. By
\cite[Remark~2.4]{MaZhengLiu2026}, the pair $(F,F')$ is an adjoint
pair. Consequently, $(F^n,(F')^n)$ is also an adjoint pair. Hence, for
all objects $X,Y\in\cB$, there is a natural isomorphism
\[
\Hom_{\cB}(F^nX,Y)
\cong
\Hom_{\cB}(X,(F')^nY).
\]
If $F^n=0$, the right-hand side vanishes for all $X$ and $Y$. Taking
$X=(F')^nY$ shows that $(F')^nY=0$, and hence $(F')^n=0$. The converse
follows similarly by taking $Y=F^nX$. Therefore
\[
F^n=0\quad\Longleftrightarrow\quad(F')^n=0,
\]
which completes the proof.
\end{proof}

The following dimension-shifting statement will be used repeatedly.

\begin{lemma}\label{lem:bounded-middle-syzygy}
Let
\[
0\longrightarrow X\longrightarrow Y\longrightarrow Z\longrightarrow0
\]
be exact and assume $\pd_{\cA}Y\leq c<\infty$. Then $
\Omega_{\cA}^{c+1}Z\simeq \Omega_{\cA}^{c}X
$ in  $\underline{\cA}$.
\end{lemma}

\begin{proof}
Lemma~\ref{lem:rotation} gives
\[
0\to\Omega Z\to X\oplus P_Z\to Y\to0.
\]
Applying Horseshoe Lemma to the above short exact sequence, we get the following exact sequences
\[
0\to \Omega^{j+1} Z\to \Omega^{j} X\oplus Q_{j} \to \Omega^{j} Y\to0
\qquad( j\geq 0).
\]
Since $\pd_{\cA}Y\leq c<\infty$,  $\Omega^{j} Y$ is a  projective or zero object, for $j\geq c$.  So $\Omega_{\cA}^{c+1}Z\simeq \Omega_{\cA}^{c}X$ up to projective.
\end{proof}

Let $T\colon\cA\to\cB$ be an exact functor. We set
\[
\pd(T)
:=
\sup\bigl\{
\pd_{\cB}T(P)
\mid
P\in\Proj\cA
\bigr\}
\in\mathbb{N}\cup\{\infty\}.
\]

\begin{lemma}\label{lem:factor-finite-pd}
Let $\cC$ be an abelian category with enough projective objects, and let
$f\colon X\to Y$ be a morphism in $\cC$ that factors through an object
$E$ with
$\pd E\leq c<\infty$. Then $\Omega^c f$ factors through a projective object. Consequently, if $T:\cA\to\cB$ is exact and $\pd(T)\leq c$,  then the assignment
$X\longmapsto\Omega_{\cB}^{c}T(X)$ 
induces an additive functor
\[
\underline{T}_c:\underline{\cA}\longrightarrow\underline{\cB}.
\]
\end{lemma}
\begin{proof}
Write $f=ba$ with $a:X\to E$ and $b:E\to Y$. Choosing projective resolutions,  we can lift $a$ and $b$ to chain maps by the  comparison theorem for projective resolutions. The induced morphism on the $c$-th syzygies factors through $\Omega^cE$, which is projective  or zero because $\pd E\leq c$. Hence $\Omega^cf$ is zero in the stable category.

If $f$ and $g$ represent the same morphism in $\underline{\cA}$, then $f-g$ factors through a projective $P$. Thus $T(f-g)$ factors through $T(P)$, whose projective dimension is at most $c$, and the first assertion shows that $\Omega^cT(f)$ and $\Omega^cT(g)$ agree in $\underline{\cB}$. Additivity and compatibility with composition follow from the comparison theorem for projective resolutions.
\end{proof}

\begin{proposition}\label{prop:stable-syzygy-comparison}
Let $T:\cA\to\cB$ be exact with $\pd(T)\leq c<\infty$. Then, for every $X\in\cA$ and every $n\geq0$, there is an  isomorphism in $\underline{\cB}$
\[
\Omega_{\cB}^{c+n}T(X)
\simeq
\Omega_{\cB}^{c}T(\Omega_{\cA}^{n}X).
\]
\end{proposition}

\begin{proof}
The assertion is immediate for $n=0$. For $n=1$, apply $T$ to the exact
sequence
\[
0\longrightarrow\Omega_{\cA}X
\longrightarrow P_X
\longrightarrow X
\longrightarrow0.
\]
Since $T$ is exact, this yields an exact sequence
\[
0\longrightarrow T(\Omega_{\cA}X)
\longrightarrow T(P_X)
\longrightarrow T(X)
\longrightarrow0.
\]
Moreover, $\pd_{\cB}T(P_X)\leq c$ 
because $P_X$ is projective in $\cA$. Lemma~\ref{lem:bounded-middle-syzygy}
therefore gives an isomorphism
\[
\Omega_{\cB}^{c+1}T(X)
\cong
\Omega_{\cB}^{c}T(\Omega_{\cA}X)
\]
in $\underline{\cB}$.

The general case follows by repeatedly applying the case $n=1$ and
using the fact that the syzygy functor preserves isomorphisms in the
stable category.
\end{proof}

\begin{proposition}\label{thm:exact-functor-transfer}
Let $T\colon\cA\to\cB$ be an exact functor between abelian categories
with enough projective objects. Then, for every $X\in\cA$ and every
$k\geq1$,
\[
k\text{-}\dell_{\cB}(T X)
\leq
\pd(T)+k\text{-}\dell_{\cA}X
\]
and
\[
k\text{-}\ddell_{\cB}(T X)
\leq
\pd(T)+k\text{-}\ddell_{\cA}X.
\]
\end{proposition}

\begin{proof}
If $\pd(T)=\infty$, both inequalities are immediate. We may therefore
assume that $
\pd(T)=c<\infty.$ Suppose first that  $k\text{-}\dell_{\cA}X=m<\infty.$ 
By definition, there exist an object $Y\in\cA$ and morphisms in
$\underline{\cA}$ exhibiting a stable retraction
\[
\Omega_{\cA}^{m}X
\longrightarrow
\Omega_{\cA}^{m+k}Y
\longrightarrow
\Omega_{\cA}^{m}X.
\]
Applying the functor $\underline{T}_c$ from
Lemma~\ref{lem:factor-finite-pd} gives a stable retraction
\[
\Omega_{\cB}^{c}T(\Omega_{\cA}^{m}X)
\longrightarrow
\Omega_{\cB}^{c}T(\Omega_{\cA}^{m+k}Y)
\longrightarrow
\Omega_{\cB}^{c}T(\Omega_{\cA}^{m}X).
\]
By Proposition~\ref{prop:stable-syzygy-comparison}, this stable
retraction may be identified with
\[
\Omega_{\cB}^{c+m}T(X)
\longrightarrow
\Omega_{\cB}^{c+m+k}T(Y)
\longrightarrow
\Omega_{\cB}^{c+m}T(X).
\]
Hence
\[
k\text{-}\dell_{\cB}(T X)\leq c+m.
\]

We now prove the second inequality. Suppose that  $k\text{-}\ddell_{\cA}X=m<\infty.$
Then there is an exact sequence
\[
0\longrightarrow C_t\longrightarrow\cdots
\longrightarrow C_0\longrightarrow X\longrightarrow0,
\qquad t\leq m,
\]
such that $(i+k)\text{-}\dell_{\cA}C_i\leq m-i$ $(0\leq i\leq t).$
Since $T$ is exact, applying $T$ gives an exact sequence
\[
0\longrightarrow T(C_t)\longrightarrow\cdots
\longrightarrow T(C_0)\longrightarrow T(X)\longrightarrow0.
\]
Applying the first inequality with $i+k$ in place of $k$, we obtain
\[
\begin{aligned}
(i+k)\text{-}\dell_{\cB}T(C_i)
&\leq
c+(i+k)\text{-}\dell_{\cA}C_i  \\
&\leq c+m-i.
\end{aligned}
\]
Since $t\leq m\leq c+m$,  the preceding exact sequence satisfies the defining conditions for
a $(c+m)$-derived delooping of $T(X)$. Therefore,
\[
k\text{-}\ddell_{\cB}(T X)\leq c+m.
\]
\end{proof}

\begin{lemma}\label{lem:simple-under-i}
Let $(\cB,\cA,i,e,l)$ be a cleft extension. If $S$ is simple in $\cB$, then $iS$ is simple in $\cA$.
\end{lemma}

\begin{proof}
Let $U\into iS$ be a nonzero subobject. Applying the faithful exact functor $e$ gives a nonzero subobject $eU\into eiS\simeq S$. Hence $eU\simeq S$. The quotient $iS/U$ is sent to zero by $e$. Since $e$ is faithful, an object whose identity morphism is sent to zero must itself be zero. Thus $U=iS$.
\end{proof}

\begin{lemma}\label{lem:length-bound}
Let $\cA$ be a length category with finitely many simple objects, and fix
an integer $k\geq1$. Suppose that
$
k\text{-}\ddell(\cA)=d<\infty.$ 
Then every nonzero object $X\in\cA$ of length $
\ell_{\cA}(X)=t$
satisfies
\[
k\text{-}\ddell_{\cA}X
\leq
t(d+1)-1.
\]
\end{lemma}

\begin{proof}
We argue by induction on $t$. If $t=1$, then $X$ is simple, and hence
\[
k\text{-}\ddell_{\cA}X\leq d=t(d+1)-1.
\]
Suppose that $t>1$. Choose a simple quotient $S$ of $X$ and write the
corresponding exact sequence as
\[
0\longrightarrow X'
\longrightarrow X
\longrightarrow S
\longrightarrow0.
\]
Then $\ell_{\cA}(X')=t-1.$
By the induction hypothesis and Proposition~\ref{prop:extension-estimate},
we obtain
\[
\begin{aligned}
k\text{-}\ddell_{\cA}X
&\leq
k\text{-}\ddell_{\cA}X'
+
k\text{-}\ddell_{\cA}S
+1\\
&\leq
\bigl((t-1)(d+1)-1\bigr)+d+1\\
&=
t(d+1)-1.
\end{aligned}
\]
This completes the induction.
\end{proof}

\begin{lemma}\label{lem:cleft-transfer}
Let $(\cB,\cA,i,e,l)$ be a cleft extension of abelian categories with
enough projective objects. Assume that $l$ is exact and that $e$
preserves projective objects.
\begin{enumerate}[$(a)$]
\item For every $X\in\cA$, every $Y\in\cB$, and every $k\geq1$,
\[
k\text{-}\ddell_{\cB}(eX)
\leq
k\text{-}\ddell_{\cA}X
\]
and
\[
k\text{-}\ddell_{\cA}(lY)
\leq
k\text{-}\ddell_{\cB}Y.
\]

\item Suppose, in addition, that $e$ admits an exact right adjoint $r$.
Then, for every $X\in\cA$,
\[
\ddell_{\cA}X
\leq
\pd(r)+\ddell_{\cB}(eX)+1.
\]
Consequently, if $\pd(r)<\infty$, then
\[
\ddell_{\cA}X<\infty
\quad\Longleftrightarrow\quad
\ddell_{\cB}(eX)<\infty.
\]
\end{enumerate}
\end{lemma}

\begin{proof}
Since $e$ is exact and preserves projective objects, we have
\(\pd(e)=0\). Proposition~\ref{thm:exact-functor-transfer}, applied to
$e$, gives $k\text{-}\ddell_{\cB}(eX)
\leq
k\text{-}\ddell_{\cA}X.
$ 
By Lemma~\ref{lem:cleft-properties}, the functor $l$ preserves
projective objects. Since $l$ is exact by assumption, we have
\(\pd(l)=0\). Applying Proposition~\ref{thm:exact-functor-transfer} to
$l$ gives
$
k\text{-}\ddell_{\cA}(lY)
\leq
k\text{-}\ddell_{\cB}Y.
$ 
This proves~$(a)$.

For~$(b)$, apply Proposition~\ref{thm:exact-functor-transfer} to the
exact functor $r\colon\cB\to\cA$. We obtain
\[
\ddell_{\cA}(reX)
\leq
\pd(r)+\ddell_{\cB}(eX).
\]
Let $\eta_X\colon X\longrightarrow reX$ 
be the unit of the adjunction $e\dashv r$. By~\eqref{eq:2.4},
$\eta_X$ is a monomorphism. Lemma~\ref{lem:subobject-estimate} therefore
gives
\[
\begin{aligned}
\ddell_{\cA}X
&\leq
\ddell_{\cA}(reX)+1\\
&\leq
\pd(r)+\ddell_{\cB}(eX)+1.
\end{aligned}
\]
If $\pd(r)<\infty$, the preceding inequality proves
\[
\ddell_{\cB}(eX)<\infty
\Longrightarrow
\ddell_{\cA}X<\infty.
\]
The converse follows from~$(a)$, applied to the ordinary derived
delooping level. This proves the equivalence.
\end{proof}

Next, we consider the case   $\operatorname{Im}G\subseteq\operatorname{Proj}\mathcal A$, where $G:\mathcal A\to\mathcal A$ is determined by the natural exact
sequence in \eqref{eq:2.2}
\[
0\longrightarrow GX\longrightarrow leX
\xrightarrow{ }X\longrightarrow0.
\]
\begin{proposition}\label{P-cleft-ddell}
Let
$
(\mathcal B,\mathcal A,i,e,l)
$
be a cleft extension of abelian categories with enough projective
objects. Assume that  $l$ is exact, $e$ preserves projective objects  and $\operatorname{Im}G\subseteq\operatorname{Proj}\mathcal A$. Then, for every $X\in\mathcal A$ and every $k\geq1$,
\[
k\text{-}\operatorname{ddell}_{\mathcal B}(eX)
\leq
k\text{-}\operatorname{ddell}_{\mathcal A}(X)
\leq
\max\bigl\{
1,
k\text{-}\operatorname{ddell}_{\mathcal B}(eX)
\bigr\}.
\]
\end{proposition}

\begin{proof}
Since $e$ is exact and preserves projective objects,
Lemma~\ref{lem:cleft-transfer}  gives
\[
k\text{-}\operatorname{ddell}_{\mathcal B}(eX)
\leq
k\text{-}\operatorname{ddell}_{\mathcal A}(X),
\]
and
\[
k\text{-}\operatorname{ddell}_{\mathcal A}(leX)
\leq
k\text{-}\operatorname{ddell}_{\mathcal B}(eX).
\]
On the other hand, the defining exact sequence of the cleft extension
is
\[
0\longrightarrow GX\longrightarrow leX
\longrightarrow X\longrightarrow0.
\]
By assumption, $GX$ is projective. Hence
Lemma~\ref{L-projective-kernel-ddell} gives
\[
k\text{-}\operatorname{ddell}_{\mathcal A}(X)
\leq
\max\bigl\{
1,
k\text{-}\operatorname{ddell}_{\mathcal A}(leX)
\bigr\}.
\]
Combining the last two inequalities, we obtain
\[
k\text{-}\operatorname{ddell}_{\mathcal A}(X)
\leq
\max\bigl\{
1,
k\text{-}\operatorname{ddell}_{\mathcal B}(eX)
\bigr\}.
\]
Therefore
\[
k\text{-}\operatorname{ddell}_{\mathcal B}(eX)
\leq
k\text{-}\operatorname{ddell}_{\mathcal A}(X)
\leq
\max\bigl\{
1,
k\text{-}\operatorname{ddell}_{\mathcal B}(eX)
\bigr\}.
\]
\end{proof}
Applying the above results, we get the following 
 comparison of the  derived delooping
levels under a cleft extension.

\begin{theorem}\label{thm:cleft-transfer}
Let $(\cB,\cA,i,e,l)$ be a cleft extension of length categories with finitely many isomorphism classes of simple objects and enough projective objects. Assume that $l$ is exact and that $e$ preserves projective objects. Set
\[
L_e
=
\max\bigl\{
\ell_{\cB}(eS)
\mid
S\in\operatorname{simp}\cA
\bigr\},
\]
where $\ell_{\cB}(eS)$ denotes the composition length of an object
$eS$ in $\cB$. 
\begin{enumerate}[(a)]
\item Suppose that $e$ admits an exact right adjoint $r$. Then
\[
\ddell(\cB)
\leq
\ddell(\cA)
\leq
\pd(r)+L_e\bigl(\ddell(\cB)+1\bigr).
\]
Moreover, if $F^n=0$ for some $n\geq1$ and $\pd(r)<\infty$, then, for every $k\geq1$,
\[
k\text{-}\ddell(\cA)<\infty
\quad\Longleftrightarrow\quad
k\text{-}\ddell(\cB)<\infty.
\]

\item  Suppose that
$
\operatorname{Im}G\subseteq\operatorname{Proj}\cA,
$ Then, for every $k\geq1$,
\[
k\text{-}\ddell(\cB)
\leq
k\text{-}\ddell(\cA)
\leq
\max\bigl\{
1,\,
L_e\bigl(k\text{-}\ddell(\cB)+1\bigr)-1
\bigr\}.
\]
In particular,
\[
k\text{-}\ddell(\cA)<\infty
\quad\Longleftrightarrow\quad
k\text{-}\ddell(\cB)<\infty.
\]
\end{enumerate}
\end{theorem}

\begin{proof} ($a$)
Let $T$ be a simple object of $\cB$. By
Lemma~\ref{lem:simple-under-i}, the object $iT$ is simple in $\cA$.
Moreover, Lemma~\ref{lem:cleft-transfer}(a) gives
\[
\ddell_{\cB}T
=
\ddell_{\cB}(eiT)
\leq
\ddell_{\cA}(iT)
\leq
\ddell(\cA).
\]
Taking the supremum over all simple objects $T$ of $\cB$, we obtain
\[
\ddell(\cB)\leq\ddell(\cA).
\]
For the upper bound, there is nothing to prove if either $\ddell(\cB)=\infty$ or $\pd(r)=\infty.$ 
We may therefore assume that $d:=\ddell(\cB)<\infty$ and $\pd(r)<\infty.$ 
Let $S$ be a simple object of $\cA$. Since $\ell_{\cB}(eS)\leq L_e,$  Lemma~\ref{lem:length-bound}, applied with $k=1$, gives
\[
\ddell_{\cB}(eS)
\leq
L_e(d+1)-1.
\]
Hence Lemma~\ref{lem:cleft-transfer}(b) yields
\[
\begin{aligned}
\ddell_{\cA}S
&\leq
\pd(r)+\ddell_{\cB}(eS)+1\\
&\leq
\pd(r)+L_e(d+1).
\end{aligned}
\]
Taking the supremum over all simple objects $S$ of $\cA$, we obtain
\[
\ddell(\cA)
\leq
\pd(r)+L_e\bigl(\ddell(\cB)+1\bigr).
\]

We now prove the later  assertion of (a) . Fix $k\geq1$. If $k\text{-}\ddell(\cA)<\infty,$ 
then Lemma~\ref{lem:cleft-transfer}(a), applied to the simple objects
$iT$ with $T\in\mathcal{S}(\cB)$, gives
\[
k\text{-}\ddell(\cB)
\leq
k\text{-}\ddell(\cA)
<\infty.
\]
Conversely, suppose that $d_k:=k\text{-}\ddell(\cB)<\infty.$
Let $S$ be a simple object of $\cA$. For each $0\leq j<n$, the defining
exact sequence of the cleft coextension gives
\[
0
\longrightarrow
(G')^j(S)
\longrightarrow
re\bigl((G')^j(S)\bigr)
\longrightarrow
(G')^{j+1}(S)
\longrightarrow
0.
\]
Since $F^n=0$, Lemma~\ref{lem:F-G-nilpotent} implies that $(G')^n=0.$ 
We first show that
\[
k\text{-}\ddell_{\cA}
re\bigl((G')^j(S)\bigr)<\infty
\qquad
(0\leq j<n).
\]
Since $\cA$ has only finitely many simple objects and $0\leq j<n$, the
objects $e(G')^j(S)$ form a finite family of finite-length objects of
$\cB$. The assumption $
k\text{-}\ddell(\cB)<\infty,$ 
together with Lemma~\ref{lem:length-bound}, therefore implies that
their $k$-derived delooping levels admit a finite uniform bound.
Since $r$ is exact and $\pd(r)<\infty$, 
the transfer estimate for $r$
then gives
\[
k\text{-}\ddell_{\cA}
re\bigl((G')^j(S)\bigr)<\infty
\qquad
(0\leq j<n).
\]
Starting with $(G')^n(S)=0$ and applying the second inequality of
Proposition~\ref{prop:extension-estimate} successively for
$j=n-1,n-2,\ldots,0$, we obtain
\[
k\text{-}\ddell_{\cA}(G')^j(S)<\infty
\qquad
(0\leq j\leq n).
\]
In particular, $
k\text{-}\ddell_{\cA}S<\infty.$ 
Since $\cA$ has only finitely many isomorphism classes of simple
objects, taking the maximum over $S\in\mathcal{S}(\cA)$ gives
\[
k\text{-}\ddell(\cA)<\infty.
\]

This completes the proof of ($a$).

($b$)
By Proposition~\ref{P-cleft-ddell}, for every $X\in\mathcal A$ and
every $k\geq1$,
\begin{equation}\label{eq:4.5}
k\text{-}\operatorname{ddell}_{\mathcal B}(eX)
\leq
k\text{-}\operatorname{ddell}_{\mathcal A}(X)
\leq
\max\bigl\{
1,
k\text{-}\operatorname{ddell}_{\mathcal B}(eX)
\bigr\}.
\end{equation}

We first prove
$k\text{-}\operatorname{ddell}\mathcal B
\leq
k\text{-}\operatorname{ddell}\mathcal A.
$
Let $T$ be a simple object of $\mathcal B$. By Lemma \ref{lem:simple-under-i}, $iT$ is
simple in $\mathcal A$.  Applying the left-hand inequality of \eqref{eq:4.5}  to $iT$ and using
$eiT\simeq T$, we obtain
\[
k\text{-}\operatorname{ddell}_{\mathcal B}T
\leq
k\text{-}\operatorname{ddell}_{\mathcal A}(iT)
\leq
k\text{-}\operatorname{ddell}\mathcal A.
\]
Taking the supremum over all simple objects $T$ of $\mathcal B$
gives
\[
k\text{-}\operatorname{ddell}\mathcal B
\leq
k\text{-}\operatorname{ddell}\mathcal A.
\]

We now assume that
$d=k\text{-}\operatorname{ddell}\mathcal B<\infty$
and prove the converse bound. Let $S$ be a simple object of
$\mathcal A$. By Lemma \ref{lem:length-bound},
 \[
k\text{-}\operatorname{ddell}_{\mathcal B}(eS)
\leq \ell_{\mathcal B}(eS)\cdot (d+1)-1
\leq
L_e \cdot (d+1)-1.
\]
Using the right-hand inequality of  \eqref{eq:4.5}, we therefore obtain
\[
k\text{-}\operatorname{ddell}_{\mathcal A}S
\leq
\max\bigl\{
1,
k\text{-}\operatorname{ddell}_{\mathcal B}(eS)
\bigr\}
\leq
\max\{1,L_e \cdot (d+1)-1\}.
\]
Taking the supremum over all simple objects $S$ of $\mathcal A$
gives
\[
k\text{-}\operatorname{ddell}\mathcal A
\leq
\max\{1,L_e \cdot (d+1)-1\}.
\]

It follows immediately that
$k\text{-}\operatorname{ddell}\mathcal A<\infty
\quad\Longleftrightarrow\quad
k\text{-}\operatorname{ddell}\mathcal B<\infty.
$
\end{proof}

We now turn to cleft extensions of module categories. For
finite-dimensional algebras, such extensions restrict to the categories
of finitely generated modules and are related to cleft coextensions as
follows.

\begin{lemma}\label{lem:module-cleft-coextension}
{\rm\cite[Remark~2.9]{MaZhengLiu2026}}
Let $\Gamma$ and $\Lambda$ be finite-dimensional algebras over a field
$\Bbbk$, and let
$(\operatorname{Mod}\text{-}\Gamma,
 \operatorname{Mod}\text{-}\Lambda,i,e,l)$ 
be a cleft extension of module categories. Then the associated functors
restrict to the categories of finitely generated modules and induce a
cleft extension
$(\operatorname{mod}\text{-}\Gamma,
 \operatorname{mod}\text{-}\Lambda,i,e,l).$
Moreover, the induced cleft extension is the upper part of a cleft
coextension.
\end{lemma}

Since the categories of finitely generated modules over
finite-dimensional algebras are length categories with finitely many isomorphism classes of simple objects and enough projective objects,
Theorem~\ref{thm:cleft-transfer} yields the following consequence.

\begin{corollary}\label{cor:cleft-findim}
Let $(\operatorname{mod}\text{-}\Gamma,
 \operatorname{mod}\text{-}\Lambda,i,e,l)$ be a cleft extension of finite-dimensional algebras over a field
$\Bbbk$. Assume that $l$ is exact, and $e$ preserves
projective modules.
\begin{enumerate}[(a)]
\item Suppose that $e$ admits an exact right adjoint $r$. Then
\[
\ddell(\Gamma)
\leq
\ddell(\Lambda)
\leq
\pd_{\Lambda}r(\Gamma)+L_e\bigl(\ddell(\Gamma)+1\bigr).
\]
Moreover, if $F^n=0$ for some $n\geq1$ and $\pd_{\Lambda}r(\Gamma)<\infty$, then, for every $k\geq1$,
\[k\text{-}\ddell\Lambda<\infty
\quad\Longleftrightarrow\quad
k\text{-}\ddell\Gamma<\infty.\]
\item  Suppose that
$
\operatorname{Im}G\subseteq\operatorname{proj}(\Lambda),
$ Then, for every $k\geq1$,
\[
k\text{-}\ddell(\Gamma)
\leq
k\text{-}\ddell(\Lambda)
\leq
\max\bigl\{
1,\,
L_e\bigl(k\text{-}\ddell(\Gamma)+1\bigr)-1
\bigr\}.
\]
In particular,
\[
k\text{-}\ddell\Lambda<\infty
\quad\Longleftrightarrow\quad
k\text{-}\ddell\Gamma<\infty.\]
\end{enumerate}
Consequently, if,  for some $k\geq1$, either $k\text{-}\ddell\Gamma<\infty$ or $k\text{-}\ddell\Lambda<\infty$, then  $
\Findim\Gamma^{\op}<\infty \ \text{and}\
\Findim\Lambda^{\op}<\infty.
$
\end{corollary}

\begin{proof}
Since $\pd(r)=
\sup\bigl\{
\pd_{\Lambda}r(P)
\mid
P\in\operatorname{proj}\text{-}\Gamma
\bigr\}=
\pd_{\Lambda}r(\Gamma)
<\infty,$  this is a direct result of Theorem~\ref{thm:cleft-transfer}. By Theorem~1.1 of Guo--Igusa \cite{GuoIgusa2025}, one has
\[
\Findim\Gamma^{\op}\leq\ddell\Gamma\leq k\text{-}\ddell\Gamma
\qquad\text{and}\qquad
\Findim\Lambda^{\op}\leq\ddell\Lambda\leq k\text{-}\ddell\Lambda.
\]
Hence the finiteness of the two derived delooping levels implies
\[
\Findim\Gamma^{\op}<\infty
\qquad\text{and}\qquad
\Findim\Lambda^{\op}<\infty.
\] \end{proof}

The arrow-removal operation and its homological properties have been
studied in \cite{EPS2022,GPS2021}. We now apply Theorem~\ref{thm:cleft-transfer} (b) to the arrow removal
operation.

\begin{theorem}[Derived-delooping arrow removal]
\label{thm:arrow-removal}
Let $Q$ be a finite quiver, and let $\Lambda=\Bbbk Q/I$ be an admissible
quotient of the path algebra $\Bbbk Q$ over a field $\Bbbk$.  Suppose that there are arrows $a_i\colon v_{e_i}\longrightarrow v_{f_i},$
 $i=1,2,\ldots,t,$
in $Q$ which do not occur in a set of minimal generators of $I$ in
$\Bbbk Q$ and satisfy
$
\Hom_{\Lambda}(e_i\Lambda,f_j\Lambda)=0\
\text{for all }i,j\in\{1,2,\ldots,t\}.
$  Define
\[
\Gamma
=
\Lambda/
\Lambda\{\overline a_i\}_{i=1}^t\Lambda,
\]
where $\overline a_i=a_i+I$.
Then, for every $k\geq1$,
\[
k\text{-}\operatorname{ddell}\Gamma
\leq
k\text{-}\operatorname{ddell}\Lambda
\leq
\max\bigl\{
1,
k\text{-}\operatorname{ddell}\Gamma
\bigr\}.
\]
\end{theorem}

\begin{proof}
Put $J=\Lambda\{\overline a_i\}_{i=1}^t\Lambda$.
 The arrow-removal construction gives the following diagram of functors:
\[
\begin{tikzcd}[column sep=5.5em]
\mathrm{mod}\text{-}\Gamma
\arrow[bend left=0]{r}{i=\operatorname{Hom}_\Gamma({}_\Lambda \Gamma_\Gamma,-)}
&
\mathrm{mod}\text{-}\Lambda
\arrow[bend left=0]{r}{e=\operatorname{Hom}_\Lambda({}_\Gamma \Lambda_\Lambda,-)}
\arrow[bend left=40]{l}{p=\operatorname{Hom}_\Lambda({}_\Gamma \Gamma_\Lambda,-)}
\arrow[bend right=40]{l}[above]{q=-\otimes_\Lambda {}_\Lambda \Gamma_\Gamma}
&
\mathrm{mod}\text{-}\Gamma.
\arrow[bend left=40]{l}{r=\operatorname{Hom}_\Gamma({}_\Lambda \Lambda_\Gamma,-)}
\arrow[bend right=40]{l}[above]{l=-\otimes_\Gamma {}_\Gamma \Lambda_\Lambda}
\end{tikzcd}
\]
By \cite[Proposition~4.6]{GPS2021}, the following statements hold:
\begin{itemize}
\item
$
(\operatorname{mod}\text{-}\Gamma,
 \operatorname{mod}\text{-}\Lambda,i,e,l)
$
is a cleft extension;
\item the functors $l$ and $r$ are exact, and $e$ preserves projective
modules;
\item $\operatorname{Im}G\subseteq\operatorname{proj}\Lambda.$
\end{itemize}
Hence the hypotheses of
Theorem~\ref{thm:cleft-transfer} (b) are satisfied. 
Since $J\subseteq\operatorname{rad}\Lambda$, every simple
$\Lambda$-module is annihilated by $J$. Therefore the simple
$\Lambda$-modules and the simple $\Gamma=\Lambda/J$-modules are
naturally identified. More precisely, if $S_v$ is the simple module
corresponding to a vertex $v$ of $Q$, then $e(S_v)\simeq S_v$ as a $\Gamma$-module. Hence $\ell_{\Gamma}(eS_v)=1$
for every simple $\Lambda$-module $S_v$, and consequently $L_e=1.$  Applying TTheorem~\ref{thm:cleft-transfer} (b), we obtain
\[
k\text{-}\operatorname{ddell}\Gamma
\leq
k\text{-}\operatorname{ddell}\Lambda
\leq
\max\bigl\{
1,
k\text{-}\operatorname{ddell}\Gamma
\bigr\}.
\]
\end{proof}
\begin{remark}\label{R-arrow-removal-kddell}
Theorem~\ref{thm:arrow-removal} shows that arrow removal
preserves the finiteness of all higher derived delooping levels.
Thus an arrow addition or removal satisfying the above hypotheses
cannot turn an algebra with finite $k$-derived delooping level into
one with infinite $k$-derived delooping level, or conversely. In fact, the only possible change in the numerical value occurs at
the boundary value zero: if
$k\text{-}\operatorname{ddell}\Gamma\geq1,$
then
$
k\text{-}\operatorname{ddell}\Lambda
=
k\text{-}\operatorname{ddell}\Gamma,
$
whereas if
$
k\text{-}\operatorname{ddell}\Gamma=0,
$
then
$
0
\leq
k\text{-}\operatorname{ddell}\Lambda
\leq1.
$
\end{remark}

The following example shows that the arrow removal operation  can simplify the estimation for  the finiteness of the derived delooping level.
\begin{example}\label{ex:GPS-arrow-removal}\cite[Example 6.3]{GPS2021}
Let $\Lambda=\Bbbk Q/I$, where $\Bbbk$ is a field and $Q$ is the quiver
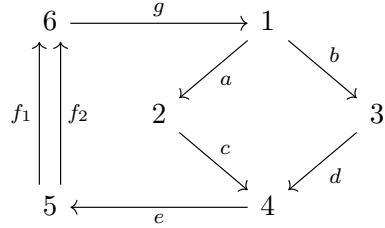
\[
\begin{tikzcd}
6 \arrow[rr,"g"] && 1 \arrow[rd,"b"] \arrow[ld,"a"] & \\
&2 \arrow[rd,"c"] & & 3\arrow[ld,"d"] \\
5\arrow[uu,shift left=4pt,"f_1"]
\arrow[uu,shift right=4pt,"f_2" ']&&  \arrow[ll,"e"] 4 & \\
\end{tikzcd}
\]
and
$I=\langle ac-bd,\, cef_1,\, de,\, ef_1g,\, f_1gb,\, ga\rangle.$
Let $\overline f_2=f_2+I\in\Lambda$, and set
\[
J=\Lambda\overline f_2\Lambda,
\qquad
\Gamma=\Lambda/J.
\]
The arrow $f_2$ does not occur in the displayed set of minimal
generators of $I$. Moreover,
$\Hom_\Lambda(e_5\Lambda,e_6\Lambda)=0.$
Thus $f_2$ satisfies the arrow-removal conditions.
We next compute the derived delooping level of $\Gamma$. The
indecomposable projective right $\Gamma$-modules have the following
Loewy series:
\[
\begin{array}{ccccccccccccc}
1 && 2 &&   3 && 4&& 5&& 6  \\
2 \ 3&& 4 && 4 && 5&&6&&1  \\
4&&5  && && 6 && 1&&3 \end{array}.
\]
It follows directly from the above list that
$\operatorname{dell}_{\Gamma}S_i=0$ for
$i=1,3,4,5,6$. Since $S_2$ is not a submodule of any projective
$\Gamma$-module, $\operatorname{ddell}_{\Gamma}S_2\neq0$. On the other
hand,
\[
\Omega_\Gamma S_2
=
\begin{array}{c}
4\\
5
\end{array}
=
\Omega_\Gamma^2S_1.
\]
Thus  $\operatorname{dell}_\Gamma S_2\leq1.$ 
Since the derived delooping level is bounded above by the ordinary
delooping level, we have
\[
\operatorname{ddell}_\Gamma S_2
\leq
\operatorname{dell}_\Gamma S_2
\leq1.
\]
Combining this with
\(\operatorname{ddell}_\Gamma S_2\neq0\), we obtain
\[
\operatorname{dell}_\Gamma S_2
=
\operatorname{ddell}_\Gamma S_2
=
1.
\]
It follows that $\operatorname{ddell}\Gamma=1.$ 
Applying Theorem~\ref{thm:arrow-removal}, we obtain
$\operatorname{ddell}\Lambda=1$. 
Furthermore,
\[
\Findim\Lambda^{\op}\leq1
\qquad\text{and}\qquad
\Findim\Gamma^{\op}\leq1.
\]

Establishing the finiteness of $\operatorname{ddell}\Lambda$ directly
would require computing successive stable syzygies for all six simple
right $\Lambda$-modules. These computations are complicated by the
additional nonzero paths involving the arrow $f_2$. By contrast, the
arrow-removal theorem reduces the problem to the quotient algebra
$
\Gamma=\Lambda/\Lambda\overline f_2\Lambda,
$
whose projective modules and syzygies are considerably easier to
describe. Thus the reduction avoids the more involved direct
calculation over $\Lambda$.
 \end{example}

\section{Derived delooping levels and recollements}
\label{sec:recollements}

In this section, we first work with abelian categories having enough
projective objects. The corresponding results for finite-dimensional
algebras are then obtained by applying the categorical results to
recollements induced by idempotents.

\begin{definition}\label{def:recollement}
{\rm\cite[Definition~2.1]{PH}}
A \emph{recollement} of abelian categories $\cA$, $\cB$, and $\cC$ is a
diagram
\begin{equation}\label{recollement}
\begin{tikzcd}[column sep=5.5em]
\cA
\arrow[r,"i"]
&
\cB
\arrow[l,bend left=40,"p"]
\arrow[l,bend right=40,swap,"q"]
\arrow[r,"e"]
&
\cC
\arrow[l,bend left=40,"r"]
\arrow[l,bend right=40,swap,"l"]
\end{tikzcd}
\end{equation}
satisfying the following conditions:
\begin{enumerate}
\item[\rm(RA1)] $(q,i,p)$ and $(l,e,r)$ are adjoint triples;
\item[\rm(RA2)] The functors $i$, $l$, and $r$ are fully faithful.

\item[\rm(RA3)] $\operatorname{Im}i=\operatorname{Ker}e.
$
\end{enumerate}
We denote this recollement by
$
\mathsf{R}(\cA,\cB,\cC).
$
\end{definition}

The following proposition collects some standard consequences of the
definition; see, for example, \cite{FP,PH,PV}.

\begin{proposition}\label{P2.2}
{\rm\cite{FP,PH,PV}}
Let $\mathsf{R}(\cA,\cB,\cC)$ be a recollement. Then the following
statements hold.
\begin{enumerate}[$(a)$]
\item The functors $i$ and $e$ are exact.

\item There are natural isomorphisms
$q l\cong0$  and $p r\cong0.$

\item The following units and counits are natural isomorphisms:
\[
\operatorname{Id}_{\cA}
\xrightarrow{\sim}
p i,
\qquad
\operatorname{Id}_{\cC}
\xrightarrow{\sim}
e l,
\]
and
\[
q i
\xrightarrow{\sim}
\operatorname{Id}_{\cA},
\qquad
e r
\xrightarrow{\sim}
\operatorname{Id}_{\cC}.
\]

\item For every $X\in\cB$, there exist objects
$Y_X,Y'_X\in\cA$ and exact sequences
\[
0
\longrightarrow
i(Y_X)
\longrightarrow
l e(X)
\longrightarrow
X
\longrightarrow
i q(X)
\longrightarrow
0
\]
and
\[
0
\longrightarrow
i p(X)
\longrightarrow
X
\longrightarrow
r e(X)
\longrightarrow
i(Y'_X)
\longrightarrow
0.
\]
The maps involving $le(X)$ and $re(X)$ are induced by the counit of
$l\dashv e$ and the unit of $e\dashv r$, respectively.
\end{enumerate}
\end{proposition}

We now give quantitative bounds for derived delooping levels in a
recollement.  For the second bound below, if $S$ is a simple object of
$\cB$, we fix an object $Y'_S\in\cA$ occurring in the second exact
sequence of Proposition~\ref{P2.2}(d).

\begin{theorem}\label{T1}
Let \eqref{recollement} be a recollement of length categories with
finitely many isomorphism classes of simple objects and enough
projective objects, and fix $k\geq1$. Set
\[ L_l = \max\bigl\{ \ell_{\cB}(lT) \mid T\in\operatorname{simp}\cC \bigr\} \quad \text{and}\quad L_r =
\max\bigl\{
\ell_{\cA}(Y'_S)
\mid
S\in\operatorname{simp}\cB
\bigr\}.
\]
Then the following statements hold.
\begin{enumerate}[$(a)$]

\item If $\pd(e)<\infty$, then $k\text{-}\ddell(\cC)
\leq
\pd(e)
+
L_l\bigl(k\text{-}\ddell(\cB)+1\bigr)-1.
$
\item If $q$ is exact, then $k\text{-}\ddell(\cA)
\leq
k\text{-}\ddell(\cB).$

\item Suppose that $r$ is exact and that $\pd(i)<\infty,$ $\pd(r)<\infty.$ 
Then
\[
\begin{aligned}
k\text{-}\ddell(\cB)
\leq
\max\Bigl\{&
\pd(i)+k\text{-}\ddell(\cA),\\
&
\pd(i)+\pd(r)+k\text{-}\ddell(\cC)
+
L_r\bigl(k\text{-}\ddell(\cA)+1\bigr)
\Bigr\}.
\end{aligned}
\]
In particular, if both
$k\text{-}\ddell(\cA)$ and $k\text{-}\ddell(\cC)$ are finite, then
$k\text{-}\ddell(\cB)$ is finite.
\end{enumerate}
\end{theorem}

\begin{proof}
\noindent
$(a)$ Let $T$ be a simple object of $\cC$. Since $el\cong\operatorname{Id}_{\cC},$ 
we have
$
T\cong e(lT).
$
By Lemma~\ref{lem:length-bound},
\[
k\text{-}\ddell_{\cB}(lT)
\leq
\ell_{\cB}(lT)
\bigl(k\text{-}\ddell(\cB)+1\bigr)-1.
\]
Proposition~\ref{thm:exact-functor-transfer} therefore gives
\[
\begin{aligned}
k\text{-}\ddell_{\cC}T
&=
k\text{-}\ddell_{\cC}(e(lT))\\
&\leq
\pd(e)+k\text{-}\ddell_{\cB}(lT)\\
&\leq
\pd(e)
+
\ell_{\cB}(lT)
\bigl(k\text{-}\ddell(\cB)+1\bigr)-1\\
&\leq
\pd(e)
+
L_l\bigl(k\text{-}\ddell(\cB)+1\bigr)-1.
\end{aligned}
\]
Taking the maximum over all simple objects of $\cC$ yields
\[
k\text{-}\ddell(\cC)
\leq
\pd(e)
+
L_l\bigl(k\text{-}\ddell(\cB)+1\bigr)-1.
\]

\medskip
\noindent
$(b)$ Since $q$ is left adjoint to the exact functor $i$, it preserves
projective objects; hence $\pd(q)=0.$ 
Let $T$ be a simple object of $\cA$. We first note that $iT$ is simple
in $\cB$. Indeed, if $X\hookrightarrow iT$ is a subobject, then the
exactness of $e$ gives
$eX\hookrightarrow eiT=0.$ 
Thus $X\in\ker e=\operatorname{Im}i$. Since $i$ is fully faithful and
exact, the subobjects of $iT$ lying in $\operatorname{Im}i$ correspond
to subobjects of $T$. Hence $X=0$ or $X=iT$. Using $qi\cong\operatorname{Id}_{\cA}$ and Proposition~\ref{thm:exact-functor-transfer}, we obtain
\[
\begin{aligned}
k\text{-}\ddell_{\cA}T=
k\text{-}\ddell_{\cA}(qiT) \leq
k\text{-}\ddell_{\cB}(iT)
\leq
k\text{-}\ddell(\cB).
\end{aligned}
\]
Taking the maximum over all simple objects $T$ of $\cA$ gives
$
k\text{-}\ddell(\cA)
\leq
k\text{-}\ddell(\cB).
$

\medskip
\noindent
$(c)$ Put $d_{\cA}=k\text{-}\ddell(\cA),$ $d_{\cC}=k\text{-}\ddell(\cC).$
Let $S$ be a simple object of $\cB$. Suppose first that $eS=0$. Then $S\in\ker e=\operatorname{Im}i.$
Hence $S\simeq iT$ for some $T\in\cA$. Since $i$ is fully faithful and
exact and $S$ is simple, $T$ is simple. Therefore,
Proposition~\ref{thm:exact-functor-transfer} gives
\[
k\text{-}\ddell_{\cB}S
=
k\text{-}\ddell_{\cB}(iT)
\leq
\pd(i)+k\text{-}\ddell_{\cA}T
\leq
\pd(i)+d_{\cA}.
\]

Assume now that $eS\neq0$. We claim that $eS$ is simple in $\cC$.
Let $0\neq u:U\longrightarrow eS$ be a monomorphism. Denote by
\[
\eta:\operatorname{Id}_{\cC}\longrightarrow el
\qquad\text{and}\qquad
\varepsilon:le\longrightarrow\operatorname{Id}_{\cB}
\]
the unit and counit of the adjunction $l\dashv e$, respectively.
Since $l$ is fully faithful, $\eta$ is a natural isomorphism. Applying $l$ to $u$ and composing with the counit gives a morphism
\[
f
:=
\varepsilon_S\circ l(u)
:
lU\longrightarrow S.
\]
We show that $f\neq0$. Using the naturality of $\eta$, we have the following commutative diagram 
\[
\begin{tikzcd}
U \arrow[r,"\eta_U", "\cong" '] \arrow[d,"u"']
&
el(U) \arrow[d,"el(u)"]
\\
eS \arrow[r,"\eta_{eS}","\cong" ']
&
el(eS),
\end{tikzcd}
\]
i.e. 
 $ el(u)\circ\eta_U = \eta_{eS}\circ u.$ 
Hence
\[
\begin{aligned}
e(f)\circ\eta_U=
e(\varepsilon_S)\circ el(u)\circ\eta_U=
e(\varepsilon_S)\circ\eta_{eS}\circ u=
u,
\end{aligned}
\]
where the last equality follows from the triangle identity $ e(\varepsilon_S)\circ\eta_{eS}=1_{eS}.$ 
Since $u\neq0$ and $\eta_U$ is an isomorphism, it follows that $e(f)\neq0,$  and therefore $f\neq0$.
 Since $S$ is simple in $\cB$, every nonzero morphism into $S$ is an
epimorphism. Thus $f:lU\to S$ is an epimorphism. As $e$ is exact,
$
e(f):elU\longrightarrow eS
$
is also an epimorphism. Since $\eta_U$ is isomorphic,  the equality
$u=e(f)\circ\eta_U$ 
shows that $u$ is epic. 
Hence $u$ is both a monomorphism and an epimorphism, and therefore an
isomorphism. Consequently, $eS$ is simple in $\cC$.

By Proposition~\ref{P2.2}(d), there is an exact sequence
\[
0
\longrightarrow
ip(S)
\longrightarrow
S
\longrightarrow
r(eS)
\longrightarrow
i(Y'_S)
\longrightarrow
0.
\]
Since $eS\neq0$, the monomorphism $ip(S)\to S$ cannot be nonzero;
otherwise, the simplicity of $S$ would imply
$ip(S)\simeq S$, and hence $eS=0$. Thus $p(S)=0$, and the preceding
sequence reduces to a short exact sequence
\[
0
\longrightarrow
S
\longrightarrow
r(eS)
\longrightarrow
i(Y'_S)
\longrightarrow
0.
\]
Since $eS$ is simple, $k\text{-}\ddell_{\cC}(eS)\leq d_{\cC},$
and therefore Proposition~\ref{thm:exact-functor-transfer} gives
\[
k\text{-}\ddell_{\cB}r(eS)
\leq
\pd(r)+d_{\cC}.
\]
On the other hand, Lemma~\ref{lem:length-bound} gives
\[
k\text{-}\ddell_{\cA}Y'_S
\leq
\ell_{\cA}(Y'_S)(d_{\cA}+1)-1,
\]
and hence
\[
\begin{aligned}
k\text{-}\ddell_{\cB}i(Y'_S)
&\leq
\pd(i)+k\text{-}\ddell_{\cA}Y'_S\\
&\leq
\pd(i)
+
\ell_{\cA}(Y'_S)(d_{\cA}+1)-1.
\end{aligned}
\]
Applying Proposition~\ref{prop:extension-estimate} to
\[
0
\longrightarrow
S
\longrightarrow
r(eS)
\longrightarrow
i(Y'_S)
\longrightarrow
0,
\]
we obtain
\[
\begin{aligned}
k\text{-}\ddell_{\cB}S
&\leq
k\text{-}\ddell_{\cB}r(eS)
+
k\text{-}\ddell_{\cB}i(Y'_S)
+1\\
&\leq
\pd(i)+\pd(r)+d_{\cC}
+
\ell_{\cA}(Y'_S)(d_{\cA}+1)\\
&\leq
\pd(i)+\pd(r)+d_{\cC}
+
L_r(d_{\cA}+1).
\end{aligned}
\]
Combining the two cases and taking the maximum over all simple objects
$S$ of $\cB$ yields
\[
\begin{aligned}
k\text{-}\ddell(\cB)
\leq
\max\Bigl\{&
\pd(i)+d_{\cA},\\
&
\pd(i)+\pd(r)+d_{\cC}
+
L_r(d_{\cA}+1)
\Bigr\}.
\end{aligned}
\]
\end{proof}

We next study recollements for which the $\cB$-relative projective global dimension of $\cA$ is at most one.
\begin{theorem}\label{thm:2}
Let \eqref{recollement} be a recollement of length categories with
finitely many isomorphism classes of simple objects,  enough
projective objects and enough injective objects.  Assume that
\[
\operatorname{pgl.dim}_{\cB}\cA
:=
\sup\bigl\{
\pd_{\cB}i(A)
\mid
A\in\cA
\bigr\}
\leq1.
\]
Let $L_l$ and $L_r$ be the constants defined in
Theorem~\ref{T1}. Then $\operatorname{gl.dim}\cA\leq1.$ Moreover, for every $k\geq1$,
\[
k\text{-}\ddell(\cC)
\leq
L_l\bigl(k\text{-}\ddell(\cB)+1\bigr)-1
\quad \text{and}\quad k\text{-}\ddell(\cB)
\leq
k\text{-}\ddell(\cC)+2L_r+3.
\]
Consequently, $k\text{-}\ddell(\cB)<\infty
\quad\Longleftrightarrow\quad
k\text{-}\ddell(\cC)<\infty.$
\end{theorem}

\begin{proof}
We first prove that $\operatorname{gl.dim}\cA\leq1.$ By \cite[Proposition~3.5]{GPS2021}, the assumption
$\operatorname{pgl.dim}_{\cB}\cA\leq1$
implies that $e$ preserves projective objects and that
$i\colon\cA\to\cB$ is a homological embedding, i.e.,  for every
$X,Y\in\cA$ and every $n\geq0$, the canonical map
\[
\Ext_{\cA}^{n}(X,Y)
\longrightarrow
\Ext_{\cB}^{n}(iX,iY)
\]
is an isomorphism. Let $A\in\cA$. By assumption,
$\pd_{\cB}i(A)\leq1.$ 
Hence
\[
\Ext_{\cB}^{n}(i(A),i(X))=0
\qquad
\text{for every }X\in\cA\text{ and every }n\geq2.
\]
The homological-embedding property therefore gives
\[
\Ext_{\cA}^{n}(A,X)
\cong
\Ext_{\cB}^{n}(i(A),i(X))
=
0
\]
for every $X\in\cA$ and every $n\geq2$. Consequently,
 $\pd_{\cA}A\leq1.$ 
Since $A$ was arbitrary, we conclude that
$\operatorname{gl.dim}\cA\leq1.$ 
In particular, $k\text{-}\ddell(\cA)\leq1.$ 

Since $e$ preserves projective objects, 
Theorem~\ref{T1}(a) therefore yields
\[
k\text{-}\ddell(\cC)
\leq
L_l\bigl(k\text{-}\ddell(\cB)+1\bigr)-1.
\]
We next establish the converse bound. The relative projective-dimension
assumption gives
\[
\pd(i)
=
\sup\bigl\{
\pd_{\cB}i(P)
\mid
P\in\Proj\cA
\bigr\}
\leq1.
\]
 Since $r$ is a right adjoint of $e$
and $e$ preserves projective objects, $r$ is exact. To apply Theorem~\ref{T1}(c), it remains to estimate $\pd(r)$.  Let $Q\in\Proj\cC$. Since $l$ is left
adjoint to the exact functor $e$, the object $P:=l(Q)$ is projective in $\cB$. Moreover, $e(P)=el(Q)\cong Q.$ 
Applying the standard recollement sequence to $P$, we obtain an exact
sequence
\[
0
\longrightarrow
ip(P)
\longrightarrow
P
\longrightarrow
r(Q)
\longrightarrow
i(A')
\longrightarrow
0
\]
for some $A'\in\cA$. Set $ K=\operatorname{Im}\bigl(P\longrightarrow r(Q)\bigr).$ 
Then there are short exact sequences
\[
0
\longrightarrow
ip(P)
\longrightarrow
P
\longrightarrow
K
\longrightarrow
0
\]
and
\[
0
\longrightarrow
K
\longrightarrow
r(Q)
\longrightarrow
i(A')
\longrightarrow
0.
\]
By the relative projective-dimension assumption, $\pd_{\cB}ip(P)\leq1,$  and $\pd_{\cB}i(A')\leq1.$
Since $P$ is projective, the first short exact sequence gives
\[
\pd_{\cB}K
\leq
\max\bigl\{
\pd_{\cB}P,\,
\pd_{\cB}ip(P)+1
\bigr\}
\leq2.
\]
The second short exact sequence then gives
\[
\pd_{\cB}r(Q)
\leq
\max\bigl\{
\pd_{\cB}K,\,
\pd_{\cB}i(A')
\bigr\}
\leq2.
\]
Taking the supremum over all $Q\in\Proj\cC$, we obtain $\pd(r)\leq2.$  Applying Theorem~\ref{T1}(c), together with $k\text{-}\ddell(\cA)\leq1,$  $\pd(i)\leq1,$ and $\pd(r)\leq2,$
gives
\[
\begin{aligned}
k\text{-}\ddell(\cB)
&\leq
\max\Bigl\{
2,\,
1+2+k\text{-}\ddell(\cC)+2L_r
\Bigr\}\\
&=
k\text{-}\ddell(\cC)+2L_r+3.
\end{aligned}
\]

The asserted equivalence of finiteness follows immediately.

\end{proof}

In what follows, all modules are right modules. We apply the above theorem to
finite-dimensional algebras.   Let $\Lambda$ be a finite-dimensional algebra, and let $e\in\Lambda$ be an
idempotent element. To avoid confusion with the idempotent $e$, we denote the
recollement functor
$e:\operatorname{mod}\Lambda
\longrightarrow
\operatorname{mod}e\Lambda e
$ by $\mathsf e$ throughout in the following.  Put
\[
\overline\Lambda=\Lambda/\Lambda e\Lambda,
\qquad E=e\Lambda e.
\]
The standard recollement of finitely generated right module categories has
\[
i=\operatorname{inc},\quad
\mathsf e (-)=\Hom_\Lambda(e\Lambda,-)\simeq -\otimes_\Lambda\Lambda e,
\quad q=-\otimes_\Lambda\overline\Lambda,
\quad r=\Hom_E(\Lambda e,-).
\]
The projective dimensions of these functors are
\[
\pd(e)=\pd_E(\Lambda e),\ \
\pd(i)=\pd_\Lambda\overline{\Lambda},
\ \
\text{and}\ \
\pd(r)
=
\pd_\Lambda\Hom_E(\Lambda e,E).
\]
Here $\overline{\Lambda}$ and
$\Hom_E(\Lambda e,E)$ are regarded as right $\Lambda$-modules in the
natural way.
Moreover, $q$ is exact when ${}_\Lambda\overline\Lambda$ is projective, and
$r$ is exact when $(\Lambda e)_E$ is projective.

For this recollement, write
\[
L_l(e)
=
\max\bigl\{
\ell_{\Lambda}(lT)
\mid
T\in\operatorname{simp}E
\bigr\}
 \quad \text{and}\quad 
L_r(e)
=
\max\bigl\{
\ell_{\overline{\Lambda}}(Y'_S)
\mid
S\in\operatorname{simp}\Lambda
\bigr\},
\]
where $Y'_S$ is determined by the standard recollement sequence
\[
0\longrightarrow ip(S)
\longrightarrow S
\longrightarrow re(S)
\longrightarrow i(Y'_S)
\longrightarrow0.
\]

\begin{corollary}[Derived-delooping vertex removal]
\label{cor:idempotent-recollement}
Let $\Lambda$ be a finite-dimensional algebra with an idempotent $e$,
and fix $k\geq1$. 
\begin{enumerate}[(a)]

\item If $\pd_E(\Lambda e)<\infty$, then
$k\text{-}\ddell E
\leq
\pd_E(\Lambda e)
+
L_l(e)\bigl(k\text{-}\ddell\Lambda+1\bigr)-1.
$

\item If ${}_\Lambda\overline{\Lambda}$ is projective, then
$
k\text{-}\ddell\overline{\Lambda}
\leq
k\text{-}\ddell\Lambda.
$

\item Suppose that $(\Lambda e)_E$ is projective and
$\pd_\Lambda\overline{\Lambda}<\infty,$ $\pd_\Lambda\Hom_E(\Lambda e,E)<\infty.$ Then
\[
\begin{aligned}
k\text{-}\ddell\Lambda
\leq
\max\Bigl\{&
\pd_\Lambda\overline{\Lambda}
+
k\text{-}\ddell\overline{\Lambda}, \pd_\Lambda\overline{\Lambda}
+
\pd_\Lambda\Hom_E(\Lambda e,E)\\
&\qquad\qquad\qquad\qquad
+
k\text{-}\ddell E 
+
L_r(e)\bigl(k\text{-}\ddell\overline{\Lambda}+1\bigr)
\Bigr\}.
\end{aligned}
\]
\item Assume that $\sup\bigl\{
\pd_\Lambda S
\mid
S\text{ is a simple }\overline{\Lambda}\text{-module}
\bigr\}
\leq1.$ 
Then
\[
k\text{-}\ddell E
\leq
L_l(e)\bigl(k\text{-}\ddell\Lambda+1\bigr)-1\ 
\text{and}\ 
k\text{-}\ddell\Lambda
\leq
k\text{-}\ddell E+2L_r(e)+3.
\]
\end{enumerate}
\end{corollary}

\begin{proof}
Parts~(a)--(c) follow directly from Theorem~\ref{T1}, since projective
modules over either outer algebra are direct summands of finite direct
sums of the corresponding regular module. Hence the displayed
projective dimensions provide precisely the required uniform bounds.

For part~(d), the assumption gives
\[
\operatorname{pgl.dim}_{\Lambda}\overline{\Lambda}
:=
\sup\bigl\{
\pd_\Lambda X
\mid
X\in\operatorname{mod}\text{-}\overline{\Lambda}
\bigr\}
\leq1.
\]
Thus the assertion follows from Theorem~\ref{thm:2}.
\end{proof}
 
We now specialize the preceding recollement results to triangular matrix algebras. In this setting, the two diagonal idempotents give rise to two canonical recollements of module categories, and the corresponding length constants can be computed explicitly.  For the lower triangular matrix algebra $
\Lambda=
\begin{pmatrix}
A&0\\
{}_BM_A&B
\end{pmatrix}$,  the two standard idempotents $e_1=
\begin{pmatrix}
1&0\\
0&0
\end{pmatrix}$ and $e_2=
\begin{pmatrix}
0&0\\
0&1
\end{pmatrix}$ induce two canonical recollements.   The description of the $\Lambda$-module  can be given, see the details in \cite[III.2]{ARS},  in terms of a triple 
\[
(X,Y,\varphi),
\qquad
\varphi:Y\otimes_BM\longrightarrow X.
\]

For $e_1$, we have $e_1\Lambda e_1\cong A,$ $\Lambda/\Lambda e_1\Lambda\cong B,$  
and hence a recollement
\[
\begin{tikzcd}
\operatorname{mod}B
\arrow[r,"i_1"]
&
\operatorname{mod}\Lambda
\arrow[l,bend left=40,"p_1"]
\arrow[l,bend right=40,swap,"q_1"]
\arrow[r,"e_1"]
&
\operatorname{mod}A
\arrow[l,bend left=40,"r_1"]
\arrow[l,bend right=40,swap,"l_1"] .
\end{tikzcd}
\]
The functors are given by
\[
i_1(Y)=(0,Y,0),
\qquad
e_1(X,Y,\varphi)=X,  \qquad
l_1(X)=(X,0,0),
\]
while
$
r_1(X)
=
\bigl(
X,\Hom_A(M,X),\operatorname{ev}
\bigr),
$
where $\operatorname{ev}:
\Hom_A(M,X)\otimes_BM\longrightarrow X$
 is the evaluation map. If $\widetilde{\varphi}:
Y\longrightarrow\Hom_A(M,X)$  denotes the morphism adjoint to $\varphi$, then
\[
q_1(X,Y,\varphi)=Y,
\qquad
p_1(X,Y,\varphi)=\ker\widetilde{\varphi}.
\]
In particular, the standard recollement sequence
\[
0\longrightarrow i_1p_1(Z)
\longrightarrow Z
\longrightarrow r_1e_1(Z)
\longrightarrow i_1(Y'_Z)
\longrightarrow0
\]
takes the explicit form
\[
0
\longrightarrow
(0,\ker\widetilde{\varphi},0)
\longrightarrow
(X,Y,\varphi)
\longrightarrow
\bigl(
X,\Hom_A(M,X),\operatorname{ev}
\bigr)
\longrightarrow
(0,\Coker\widetilde{\varphi},0)
\longrightarrow0.
\]
Thus
$Y'_Z\cong\Coker\widetilde{\varphi}.
$

For $e_2$, we have
$e_2\Lambda e_2\cong B,$   $\Lambda/\Lambda e_2\Lambda\cong A,$
and hence a recollement
\[
\begin{tikzcd}
\operatorname{mod}A
\arrow[r,"i_2"]
&
\operatorname{mod}\Lambda
\arrow[l,bend left=40,"p_2"]
\arrow[l,bend right=40,swap,"q_2"]
\arrow[r,"e_2"]
&
\operatorname{mod}B
\arrow[l,bend left=40,"r_2"]
\arrow[l,bend right=40,swap,"l_2"] .
\end{tikzcd}
\]
Here
\[
i_2(X)=(X,0,0),
\qquad
e_2(X,Y,\varphi)=Y,
\qquad
l_2(Y)
=
\bigl(
Y\otimes_BM,\,
Y,\,
\operatorname{id}_{Y\otimes_BM}
\bigr),
\]
\[
r_2(Y)=(0,Y,0), 
\qquad
p_2(X,Y,\varphi)=X,
\qquad
q_2(X,Y,\varphi)=\Coker\varphi.
\]
In this case, the second standard recollement sequence reduces to the
short exact sequence
\[
0
\longrightarrow
(X,0,0)
\longrightarrow
(X,Y,\varphi)
\longrightarrow
(0,Y,0)
\longrightarrow0.
\]
Consequently, $Y'_Z=0$ for every $Z\in\operatorname{mod}\Lambda$ in the recollement induced
by $e_2$.

\begin{lemma}\label{lem:triangular-length-constants}
Let $\Lambda=
\begin{pmatrix}
A&0\\
{}_BM_A&B
\end{pmatrix}$ 
be a finite-dimensional lower triangular matrix algebra, and let $e_1=
\begin{pmatrix}
1&0\\
0&0
\end{pmatrix},$ $e_2=
\begin{pmatrix}
0&0\\
0&1
\end{pmatrix}.
$ For $i=1,2$, let $L_l(e_i)$ and $L_r(e_i)$ be the constants associated
with the idempotent recollement determined by $e_i$. Then
\[
L_l(e_1)=1, \qquad 
L_r(e_1)
=
\max\bigl\{
\ell_B(\Hom_A(M,T))
\mid
T\in\operatorname{simp}A
\bigr\},
\]
\[
L_l(e_2)
=
1+
\max\bigl\{
\ell_A(T\otimes_BM)
\mid
T\in\operatorname{simp}B
\bigr\},
\qquad L_r(e_2)=0.
\]
\end{lemma}

\begin{proof}
We first consider $e_1$. For every $T\in\operatorname{simp}A$, $l_1(T)= (T,0,0)$ is a simple $\Lambda$-module. 
Therefore  $\ell_\Lambda(l_1(T))=1,$ and hence $L_l(e_1)=1.$

 For a $\Lambda$-module $(X,Y,\varphi),$ let $\widetilde{\varphi}:
Y\longrightarrow\Hom_A(M,X)$ 
be the morphism adjoint to $\varphi$. The standard recollement sequence yields
$Y'_{(X,Y,\varphi)}
\cong
\operatorname{Coker}\widetilde{\varphi}.
$ 
If $S=(T,0,0)$ with $T\in\operatorname{simp}A$, then $Y'_S
\cong
\Hom_A(M,T),$  whereas if $S=(0,U,0)$ with $U\in\operatorname{simp}B$, then $Y'_S=0.$ 
Thus
\[
L_r(e_1)
=
\max\bigl\{
\ell_B\Hom_A(M,T)
\mid
T\in\operatorname{simp}A
\bigr\}.
\]

We next consider $e_2$. For every $T\in\operatorname{simp}B$, $l_2(T)
=
\bigl(
T\otimes_BM,\,
T,\,
\operatorname{id}
\bigr).$
There is a short exact sequence
\[
0
\longrightarrow
(T\otimes_BM,0)
\longrightarrow
l_2(T)
\longrightarrow
(0,T)
\longrightarrow
0.
\]
Since $(0,T)$ is a simple $\Lambda$-module and $\ell_\Lambda(T\otimes_BM,0)
=
\ell_A(T\otimes_BM),$ we obtain $\ell_\Lambda(l_2(T))
=
\ell_A(T\otimes_BM)+1.$ 
Therefore
\[
L_l(e_2)
=
1+
\max\bigl\{
\ell_A(T\otimes_BM)
\mid
T\in\operatorname{simp}B
\bigr\}.
\]
Finally, for every $\Lambda$-module $(X,Y,\varphi)$,  the standard
recollement sequence yields  $Y'_{(X,Y,\varphi)}=0$,  and consequently $L_r(e_2)=0.$
\end{proof}

Using the explicit constants above, the general recollement estimates give the following triangular reduction.

\begin{corollary}[Triangular reduction]
\label{C1}
Let
$
\Lambda
=
\begin{pmatrix}
A&0\\
{}_BM_A&B
\end{pmatrix}
$
be a finite-dimensional lower triangular matrix algebra. Set
$L=1+
\max\left\{
\ell_A(T\otimes_BM)
\mid
T\in\operatorname{simp}B
\right\}
$ 
and fix
$k\geq1$.  Then:
\begin{enumerate}[$(a)$]
\item $k\text{-}\ddell B\leq L\bigl(k\text{-}\ddell\Lambda+1\bigr)-1.$  
\item If $\pd_A M<\infty$, then $ k\text{-}\ddell A \leq \pd_A M + k\text{-}\ddell\Lambda.$
\item If $\pd_A M<\infty$, then
$\begin{aligned}
k\text{-}\ddell\Lambda
\leq
\max\Bigl\{
k\text{-}\ddell A,\pd_A M+1+
k\text{-}\ddell B \Bigr\}.
\end{aligned}
$
\end{enumerate}
In particular, if $\operatorname{gl.dim}A\leq1, $ $k\text{-}\operatorname{ddell}\Lambda
\leq
k\text{-}\operatorname{ddell}B+2.
$
\end{corollary}

\begin{proof}
Let $e_1=
\begin{pmatrix}
1&0\\
0&0
\end{pmatrix},$  $e_2=
\begin{pmatrix}
0&0\\
0&1
\end{pmatrix}.$ 
Then $e_1\Lambda e_1\cong A,$ $e_2\Lambda e_2\cong B.$ Moreover, as right $A$-modules,
$
(\Lambda e_1)_A
\cong
A_A\oplus M_A,$ and therefore $\pd_A(\Lambda e_1)=\pd_A M.$ 
On the other hand,
$
(\Lambda e_2)_B\cong B_B,
$
so $\Lambda e_2$ is projective as a right $B$-module. We will also use the exact sequence of right $\Lambda$-modules
\[
0
\longrightarrow
M\otimes_A e_1\Lambda
\longrightarrow
e_2\Lambda
\longrightarrow
B
\longrightarrow
0,
\]
where $B$ is regarded as a right $\Lambda$-module via the canonical
projection. The bimodule ${}_Ae_1\Lambda_\Lambda$ is projective as a left
$A$-module, and the functor $-\otimes_A e_1\Lambda
\colon
\operatorname{mod}\text{-}A
\longrightarrow
\operatorname{mod}\text{-}\Lambda$ 
is exact and sends projective right $A$-modules to projective right
$\Lambda$-modules. Consequently, $\pd_\Lambda(M\otimes_Ae_1\Lambda)
\leq
\pd_A M.$ 
The preceding exact sequence therefore gives
\[
\pd_\Lambda B
\leq
\pd_A M+1.
\]

For~$(a)$, apply
Corollary~\ref{cor:idempotent-recollement}$(a)$ with $e=e_2$ and $L_l(e_2)=L$. Since  
$(\Lambda e_2)_B\cong B_B,
$
we have $\pd_B(\Lambda e_2)=0$, and hence
\[
k\text{-}\ddell B
\leq
L\bigl(k\text{-}\ddell\Lambda+1\bigr)-1.
\]

For~$(b)$, apply
Corollary~\ref{cor:idempotent-recollement}$(a)$ with $e=e_1$ and $L_l(e_1)=1$. Since $e_1\Lambda e_1\cong A$ and $\pd_A(\Lambda e_1)=\pd_A M<\infty,$
we obtain
\[
k\text{-}\ddell A
\leq
\pd_A M + k\text{-}\ddell\Lambda.
\]

For~$(c)$, take $e=e_2$. There are algebra isomorphisms $\Lambda/\Lambda e_2\Lambda\cong A,$ $e_2\Lambda e_2\cong B.$
Moreover, as right $\Lambda$-modules, $\Lambda/\Lambda e_2\Lambda
\cong
e_1\Lambda,$  so $\pd_\Lambda(\Lambda/\Lambda e_2\Lambda)=0.$ We also have $(\Lambda e_2)_B\cong B_B.$
Hence
$\Hom_B(\Lambda e_2,B)\cong B$ as right $\Lambda$-modules. Therefore
\[
\pd_\Lambda\Hom_B(\Lambda e_2,B)
=
\pd_\Lambda B
\leq
\pd_A M+1.
\]
Applying Corollary~\ref{cor:idempotent-recollement}(c) with $L_r(e_2)=0$ gives
\[
\begin{aligned}
k\text{-}\ddell\Lambda
\leq
\max\Bigl\{
k\text{-}\ddell A,\pd_A M+1+k\text{-}\ddell B\Bigl\}.
\end{aligned}
\]

Finally, if $\operatorname{gl.dim}A\leq1, $ then $\pd_A M\leq1$ and $\ddell A\leq1$. By ($c$), $k\text{-}\operatorname{ddell}\Lambda
\leq
k\text{-}\operatorname{ddell}B+2.
$

\end{proof}

We illustrate Corollary~\ref{C1} with a simple example in which the estimate in part~$(c)$ is attained.
\begin{example}
Let $r\geq1$ and consider the lower triangular matrix algebra
\[
\Lambda_r=
\begin{pmatrix}
\Bbbk&0\\
\Bbbk^r&\Bbbk
\end{pmatrix}.
\]
Equivalently, $\Lambda_r$ is the path algebra of the acyclic quiver
\[
\begin{tikzcd}
2
\arrow[r,shift left=6pt,"\alpha_1"]
\arrow[r,shift left=2pt,"\alpha_2"]
\arrow[r,shift right=2pt,"\cdots" description]
\arrow[r,shift right=6pt,"\alpha_r"']
&
1 .
\end{tikzcd}
\]
Here
$A=B=\Bbbk,$ ${}_BM_A=\Bbbk^r.$
Since $A$ and $B$ are semisimple,
$
k\text{-}\ddell A
=
k\text{-}\ddell B
=
0
$
for every $k\geq1$, and $\pd_A M=0.$ Hence Corollary~\ref{C1}(c) gives
\[
k\text{-}\ddell\Lambda_r
\leq
\max\{0,1\}
=
1.
\]

On the other hand, $\Lambda_r$ is a non-semisimple hereditary algebra.
Let $S$ be its non-projective simple module. Since every positive
syzygy over a hereditary algebra is projective, $S$ cannot be a
stable direct summand of a positive syzygy. Hence $ k\text{-}\dell_{\Lambda_r}S>0 $
for every $k\geq1$. Consequently,
$
k\text{-}\ddell\Lambda_r\geq1.
$
Therefore
\[
k\text{-}\ddell\Lambda_r=1
\qquad
\text{for every }k\geq1.
\]
Thus this example  makes  the estimate in
Corollary~\ref{C1}(c) is sharp.
\end{example}

We next apply Corollary~\ref{C1} to the homological heart of a bound quiver
algebra. The homological heart of a quotient of a path algebra is introduced
in \cite{GM2017}.  Let $\Lambda=\Bbbk Q/I$ be an admissible quotient of the path algebra
$\Bbbk Q$ over a field $\Bbbk$. Define
\[
X
=
\bigl\{
v\in Q_0
\mid
v\text{ lies on a nontrivial oriented cycle in }Q
\bigr\}
\]
and
\[
Y
=
\bigl\{
v\in Q_0
\mid
v\text{ lies on a path whose initial and terminal vertices belong to }
X
\bigr\}.
\]
The \emph{homological heart} $H=H(Q)$ is the full subquiver of $Q$ with
vertex set
\[
H_0=Y.
\]
We also consider the full subquivers $H^+$, $H^-$, and $H^o$ determined
by the following vertex sets:
\begin{align*}
H_0^+
&=
\bigl\{
v\in Q_0\setminus H_0
\mid
\text{there exists a path }H_0\rightsquigarrow v
\bigr\},\\
H_0^-
&=
\bigl\{
v\in Q_0\setminus H_0
\mid
\text{there exists a path }v\rightsquigarrow H_0
\bigr\},\\
H_0^o
&=
\bigl\{
v\in Q_0\setminus H_0
\mid
\text{there is no path }H_0\rightsquigarrow v
\text{ and no path }v\rightsquigarrow H_0
\bigr\}.
\end{align*}

Let $e$, $e^+$, $e^-$, and $e^o$ be the sums of the vertex idempotents
corresponding to the vertices in $H_0$, $H_0^+$, $H_0^-$, and $H_0^o$,
respectively. Then
\[
1_\Lambda=e^++e^o+e+e^-.
\]
With respect to this decomposition, $\Lambda$ has the block form
\[
\Lambda
\cong
\begin{pmatrix}
e^+\Lambda e^+&0&0&0\\
e^o\Lambda e^+&e^o\Lambda e^o&0&0\\
e\Lambda e^+&0&e\Lambda e&0\\
e^-\Lambda e^+&e^-\Lambda e^o&e^-\Lambda e&e^-\Lambda e^-
\end{pmatrix}.
\]
By \cite[Proposition~5.1(4)]{GM2017}, the full subquiver of $Q$ with
vertex set
\[
H_0^+\cup H_0^-\cup H_0^o
\]
has no oriented cycles. Consequently, the algebras
\[
\begin{pmatrix}
e^+\Lambda e^+&0\\
e^o\Lambda e^+&e^o\Lambda e^o
\end{pmatrix}
\qquad\text{and}\qquad
e^-\Lambda e^-
\]
have finite global dimension.  By two applications of the Corollary~\ref{C1}, one obtains explicit upper bounds relating $k\text{-}\ddell\Lambda$ and $k\text{-}\ddell(e\Lambda e)$.
For the present application, we only record the resulting finiteness equivalence.
\begin{corollary}
\label{cor:ddell-homological-heart}
Let $\Lambda=\Bbbk Q/I$ be an admissible quotient of a path algebra,
let $H=H(Q)$ be its homological heart, and let $e$ and $e^-$ be the sums
of the vertex idempotents corresponding to $H_0$ and $H_0^-$,
respectively. If
$\pd_{e\Lambda e}(e^-\Lambda e)<\infty,$
then, for every $k\geq1$,
\[
k\text{-}\ddell\Lambda<\infty
\quad\Longleftrightarrow\quad
k\text{-}\ddell(e\Lambda e)<\infty.
\]
\end{corollary}

\begin{proof}
Set
\[
A
=
\begin{pmatrix}
e^+\Lambda e^+&0\\
e^o\Lambda e^+&e^o\Lambda e^o
\end{pmatrix},
\quad
B
=
\begin{pmatrix}
e\Lambda e&0\\
e^-\Lambda e&e^-\Lambda e^-
\end{pmatrix},
\quad 
\text{and}
\quad
M
=
\begin{pmatrix}
e\Lambda e^+&0\\
e^-\Lambda e^+&e^-\Lambda e^o
\end{pmatrix}.
\]
Then $\Lambda
\cong
\begin{pmatrix}
A&0\\
{}_BM_A&B
\end{pmatrix}.
$
As observed above, $\operatorname{gl.dim}A<\infty$ and $\operatorname{gl.dim}(e^-\Lambda e^-)<\infty$.
Hence $k\text{-}\ddell A<\infty,$ $k\text{-}\ddell(e^-\Lambda e^-)<\infty.$ 
Moreover, $M$ is a finitely generated right $A$-module. Since $A$ has
finite global dimension,
$\pd_A M<\infty.$
Corollary~\ref{C1}, applied to
$\Lambda
\cong
\begin{pmatrix}
A&0\\
M&B
\end{pmatrix},
$
therefore gives
\begin{equation}\label{eq:heart-first-reduction}
k\text{-}\ddell\Lambda<\infty
\quad\Longleftrightarrow\quad
k\text{-}\ddell B<\infty.
\end{equation}
By assumption,
$\pd_{e\Lambda e}(e^-\Lambda e)<\infty$.
A second application  of Corollary~\ref{C1}, together with
$k\text{-}\ddell e^-\Lambda e^-<\infty$, gives
\[
k\text{-}\ddell B<\infty
\quad\Longleftrightarrow\quad
k\text{-}\ddell e\Lambda e<\infty.
\tag{5.3}
\]
Combining  (5.2) and  (5.3) proves the assertion.
\end{proof}

Finally, we illustrate the applications of vertex reduction to ddell with several examples. Throughout, modules are right modules and paths are composed from left to right.

\begin{example}\label{ex:vertex-reduction-arrow-core}
Let
$\Lambda=\Bbbk Q/I,$
where \(Q\) is the quiver
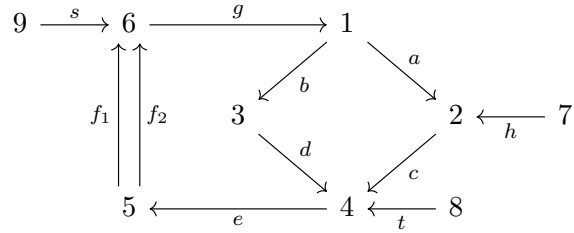
\[
\begin{tikzcd}
9\arrow[r,"s"]&6 \arrow[rr,"g"] && 1 \arrow[rd,"a"] \arrow[ld,"b"] & \\
&&3 \arrow[rd,"d"] & & 2\arrow[ld,"c"] &7\arrow[l,"h"]\\
&5\arrow[uu,shift left=4pt,"f_1"]
 \arrow[uu,shift right=4pt,"f_2"']
&&4\arrow[ll,"e"]& 8\arrow[l,"t"]
\end{tikzcd}
\]
and
\[
I=
\left\langle
ac-bd,\,
cef_1,\,
de,\,
ef_1g,\,
f_1gb,\,
ga
\right\rangle.
\]
Put
\[
\varepsilon=e_1+e_2+e_3+e_4+e_5+e_6
\]
and set $A=\varepsilon\Lambda\varepsilon.$
Then \(A\) is the six-vertex algebra $A\cong
\Bbbk Q'/I,$ 
where \(Q'\) is the full subquiver of \(Q\) supported on
\(\{1,2,3,4,5,6\}\).  Since 
$\pd_\Lambda S_7=\pd_\Lambda S_8=\pd_\Lambda S_9=1.$
Apply
Corollary~\ref{cor:idempotent-recollement}(d) to remove the vertices
\(7,8,9\). We have for every
\(k\geq1\),
\[
k\text{-}\ddell\Lambda<\infty
\quad\Longleftrightarrow\quad
k\text{-}\ddell A<\infty.
\tag{5.5}
\]
The corner algebra \(A=\varepsilon\Lambda\varepsilon\) is precisely the
six-vertex algebra considered in
\cite[Example~6.3]{GPS2021} and studied in
Example~\ref{ex:GPS-arrow-removal}. 
It follows from \((5.5)\),  with $\ddell A=1$, that
$\ddell\Lambda<\infty.$
\end{example}

\begin{example}\cite[Example 1]{XI}
  Let \(\Lambda=\Bbbk Q/I\), where
\[
\begin{tikzcd}[column sep=3.7em,row sep=2.8em]
 &1 \arrow[dl,bend right=13,"\gamma"']
    \arrow[dl,bend left=13,"\beta"]
    \arrow[dr,"\eta"]&\\
2 \arrow[loop left,"\alpha"]\arrow[dr,"\delta"']
 &&3\arrow[dl,"\xi"]&5\arrow[l]\\
 &4&&
\end{tikzcd}
\]
and
$ I=\langle
       \alpha^3,\,\alpha\delta,\,\beta\delta,
       \,\eta\xi-\gamma\delta
    \rangle. $
We use a   reduction to the homological heart to prove that
\[
k\text{-}\ddell\Lambda<\infty
\qquad
\text{for every }k\geq1.
\]

The only vertex lying on a nontrivial oriented cycle is vertex \(2\).
Moreover, no path leaving vertex \(2\) can return to it. Hence the
homological heart of \(Q\) consists only of vertex \(2\):
$
H_0=\{2\}.
$
The remaining vertices decompose as
\[
H_0^-=\{1\},
\qquad
H_0^+=\{4\},
\qquad
H_0^\circ=\{3,5\}.
\]
Put $e=e_2,$ $e^-=e_1,$ and  $A=e\Lambda e.$ 
The algebra supported on the homological heart is
\[
A
\cong
\Bbbk[\alpha]/(\alpha^3).
\]
If \(T\) denotes its unique simple module, then
$ \Omega T\cong\begin{array}{ccccccccccccc}
2  \\
2  \end{array},$ and $ \Omega^2T\cong 2 \cong T. $ 
Thus \(T\) is two-periodic.  For every \(k\geq1\), it is a \(k\)-th
syzygy of either 2  or $\begin{array}{ccccccccccccc}
2  \\
2  \end{array}$, according to the parity of
\(k\).  Directly from the definition, $ k\text{-}\ddell\bigl(\Bbbk[\alpha]/(\alpha^3)\bigr)=0.$ 
Since $ e^-\Lambda e=e_1\Lambda e_2
 \cong A^{\oplus2},$  $\pd_A(e^-\Lambda e)=0.$ 
The homological-heart reduction therefore gives directly
\[
 k\text{-}\ddell\Lambda<\infty
 \quad\Longleftrightarrow\quad
 k\text{-}\ddell A<\infty.
\]
Since \(k\text{-}\ddell A=0\), we conclude that
$k\text{-}\ddell\Lambda<\infty$ for every $k\geq1.$
\end{example}

\begin{example}[Reduction to the homological heart]\label{ex:seven-vertex-homological-heart}
Let \(\Lambda=\Bbbk Q/I\), where \(Q\) is the quiver
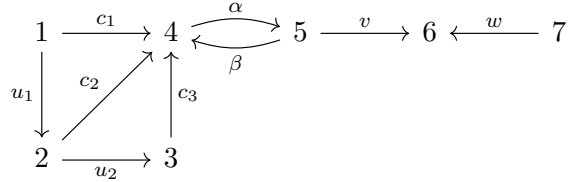
\[
\begin{tikzcd}[column sep=3.1em,row sep=3em]
1
 \arrow[d,"u_1"']
 \arrow[r,"c_1"]
&
4
 \arrow[r,bend left=18,"\alpha"]
&
5
 \arrow[l,bend left=18,"\beta"]
 \arrow[r,"v"]
&
6
&
7\arrow[l,"w"']
\\
2
 \arrow[r,"u_2"']
 \arrow[ur,"c_2"]
&
3\arrow[u,"c_3"']
\end{tikzcd}
\]
and
$I=
\left\langle
\alpha\beta,\,
\beta\alpha,\,
u_1c_2,\,
u_2c_3
\right\rangle.$ 
We apply Corollary~\ref{cor:ddell-homological-heart} to prove 
\[k\text{-}\ddell\Lambda<\infty\ \ \text{for every} \ \ k\geq1.\]

The only vertices lying on nontrivial oriented cycles are \(4\) and
\(5\). Moreover, no path leaving these two vertices can return to
either of them. Hence
\[
H_0=\{4,5\},
\qquad
H_0^-=\{1,2,3\},
\qquad
H_0^+=\{6\},
\qquad
H_0^\circ=\{7\}.
\]
Put $e=e_4+e_5,$ $e^-=e_1+e_2+e_3$, $e^+=e_6,$ $e^\circ=e_7,$ and $A=e\Lambda e.$ 
The algebra supported on the homological heart is
\[
A
\cong
\Bbbk
\left(
4\mathrel{\substack{\xrightarrow{\alpha}\\[-1mm]
                    \xleftarrow[\beta]{}}}5
\right)
\big/
\langle\alpha\beta,\beta\alpha\rangle.
\]
Let \(T_4\) and \(T_5\) denote its simple right
modules. Their projective covers give $\Omega_A T_4\cong T_5,$
$\Omega_A T_5\cong T_4.$  It follows directly from the definition that
\[
 k\text{-}\operatorname{dell}_A T_4
 =k\text{-}\operatorname{dell}_A T_5=0,
 \qquad
 k\text{-}\ddell A=0.
\]
The algebra supported on \(H^+\cup H^\circ\) is
\[
D_0
=
(e^++e^\circ)\Lambda(e^++e^\circ)
\cong
\Bbbk(7\xrightarrow{w}6),
\]
whereas the algebra supported on \(H^-\) is
\[
D_-
=
e^-\Lambda e^-
\cong
\Bbbk(1\xrightarrow{u_1}2\xrightarrow{u_2}3).
\]
Thus $\gldim D_0=\gldim D_-=1.$ 
Since 
$e^-\Lambda e\cong(e_4A)^{\oplus3}$, we obtain $\pd_A(e^-\Lambda e)=0.$
All the hypotheses of
Corollary~\ref{cor:ddell-homological-heart} are now satisfied. We conclude that
\[
k\text{-}\ddell\Lambda<\infty
\qquad
\text{for every }k\geq1.
\]

Thus a direct calculation involving all seven simple \(\Lambda\)-modules
is replaced by a periodicity calculation for the two simple modules in
the homological heart and the projectivity of \(e^-\Lambda e\).\end{example}


\begin{thebibliography}{99}


\bibitem{APT}
M. Auslander, M. I. Platzeck and G. Todorov,
\emph{Homological theory of idempotent ideals},
Trans. Amer. Math. Soc. \textbf{332} (1992), 667--692.


\bibitem{ARS}
M. Auslander, I. Reiten and S. O. Smal\o,
\emph{Representation Theory of Artin Algebras},
Corrected reprint of the 1995 original, Cambridge Studies in Advanced Mathematics, Vol. 36,
Cambridge University Press, Cambridge, 1997.

\bibitem{BLM}
M. Barrios, M. Lanzilotta and G. Mata,
\emph{Delooping levels},
Algebr. Represent. Theory \textbf{29} (2026), 203--219.


\bibitem{Beligiannis2000}
A. Beligiannis,
\emph{Cleft extensions of abelian categories and applications to ring theory},
Comm. Algebra \textbf{28} (2000), 4503--4546.


 \bibitem{BBD} \ A.  Beilinson, J.  Bernstein   and  P. Deligne, Faisceaux Pervers,  Soc. Math. France, vol. 100 (1982). 


\bibitem{CX}
H. Chen and C. Xi,
\emph{Recollements of derived categories III: finitistic dimensions},
J. Lond. Math. Soc. \textbf{95} (2017), 633--658.


\bibitem{EPS2022}
K. Erdmann, C. Psaroudaskis and \O. Solberg,
\emph{Homological invariants of the arrow removal operation}
Represent. Theory \textbf{26} (2022), 370-387.

\bibitem{EGPS2025}K. Erdmann, O. Giatagantzidis, C. Psaroudaskis and \O. Solberg,\emph{Monomial arrow removal and the finitistic dimension conjecture},arXiv:2506.23747, 2025.

\bibitem{FP} V. Franjou  and T. Pirashvili, 
\emph{Comparison of abelian categories recollements},
 Doc. Math.  \textbf{9} (2004), 41–56.


\bibitem{Gabriel1973} 
P. Gabriel, 
 \emph{Indecomposable Representations II. Symposia Mathematica}, 
 vol. XI (Convegno di Algebra Commutativa, INDAM, Rome, 1971), Academic Press, London  (1973), 81–104. 


\bibitem{GLP}
M. A. Gatica, M. Lanzilotta and M. I. Platzeck,
\emph{Idempotent ideals and the Igusa--Todorov functions},
Algebr. Represent. Theory \textbf{20} (2017), 275--287.

\bibitem{G1}
V. G\'elinas,
\emph{The finitistic dimension of an Artin algebra with radical square zero},
Proc. Amer. Math. Soc. \textbf{149} (2021), 5001--5012.


\bibitem{Gelinas2022}
V. G\'elinas,
\emph{The depth, the delooping level and the finitistic dimension},
Adv. Math. \textbf{394} (2022), 108052.


\bibitem{Giatagantzidis2025}
O. Giatagantzidis
\emph{Arrow reductions for the finitistic dimension conjecture},
arXiv:2507.12978,2025



\bibitem{GM2017}
E. L. Green and E.  N. Marcos,
\emph{Convex subquivers and the   finitistic dimension},
Illinois J. Math.  \textbf{61} (2017), no. 3-4, 385-397.


\bibitem{GPS2021}
E. L. Green, C.  Psaroudaskis and  \O. Solberg,
\emph{Reduction techniques for the finitistic dimension},
Trans. Am. Math. Soc. \textbf{374} (2021), 6839-6879.





\bibitem{Guo2025}
R. Guo,
\emph{Symmetry of derived delooping level},
Algebr. Represent. Theory \textbf{28} (2025), 1125--1137.

\bibitem{GuoIgusa2025}
R. Guo and K. Igusa,
\emph{Derived delooping levels and finitistic dimension},
Adv. Math. \textbf{464} (2025), 110152.



\bibitem{IT}
K. Igusa and G. Todorov,
\emph{On the finitistic global dimension conjecture for Artin algebras},
Fields Inst. Commun. \textbf{45} (2005), 201--204.

\bibitem{KershawRickard2024}
L. Kershaw and J. Rickard,
\emph{A finite-dimensional algebra with infinite delooping level},
Ann. Represent. Theory \textbf{1} (2024), no. 1, 61--65.


\bibitem{K2026}
P. Kostas,
\emph{Cleft extensions of  rings and singularity categories },
J.  Algebra \textbf{685} (2026), 160-224.

\bibitem{KP2025}
P. Kostas and C. Psaroudakis,
\emph{Injective generation for graded rings},
J. Pure Appl. Algebra \textbf{229} (2025), no. 7, 107960.


\bibitem{K}
S. K\"onig,
\emph{Tilting complexes, perpendicular categories and recollements of derived categories of rings},
J. Pure Appl. Algebra \textbf{73} (1991), 211--232.



\bibitem{LiangMaYang2025}
L. Liang, Y. Ma and G. Yang,
\emph{Homological invariant properties under cleft extensions},
arXiv:2506.02691, 2025.


\bibitem{MaZhengLiu2026}
Y. Ma, J. Zheng and Y.-Z. Liu,
\emph{Three homological invariants under cleft extensions},
J. Algebra \textbf{694} (2026), 287--323.


\bibitem{PX}
S. Pan and C. Xi,
\emph{Finiteness of finitistic dimension is invariant under derived equivalences},
J. Algebra \textbf{322} (2009), 21--24.

\bibitem{PH}\  C. Psaroudakis,  
\emph{Homological theory of recollements of abelian categories},
  J. Algebra \textbf{398} (2014): 63-110.  

\bibitem{PV}
C. Psaroudakis   and  J.  Vit$\acute{\mathrm{o}}$ria, 
\emph{Recollements of module categories}, 
 Appl. Categ. Struct. 22  (2014): 579-593. 


\bibitem{QH}
Y. Qin and Y. Han,
\emph{Reducing homological conjectures by $n$-recollements},
Algebr. Represent. Theory \textbf{19} (2016), 377--395.

\bibitem{R}
C. M. Ringel,
\emph{The finitistic dimension of a Nakayama algebra},
J. Algebra \textbf{576} (2021), 95--145.

\bibitem{S}
E. Sen,
\emph{Delooping level of Nakayama algebras},
Arch. Math. \textbf{117} (2021), 141--146.

\bibitem{JW}
J. Wei,
\emph{Finitistic dimension and Igusa--Todorov algebras},
Adv. Math. \textbf{222} (2009), 2215--2226.




\bibitem{XI}\ C. Xi, On the finitistic dimension conjecture I: Related to representation-finite algebras, J. Pure Appl. Algebra 193 (2004): 287-305.

\bibitem{XII}
C. Xi,
\emph{On the finitistic dimension conjecture II: related to finite global dimension},
Adv. Math. \textbf{201} (2006), 116--142.

\bibitem{XIII}
C. Xi,
\emph{On the finitistic dimension conjecture III: related to the pair $eAe\subseteq A$},
J. Algebra \textbf{319} (2008), 3666--3688.



\bibitem{XiZhang2026}
C. Xi and J. Zhang,
\emph{New invariants of stable equivalences and Auslander--Reiten conjecture},
Math. Ann. \textbf{394} (2026), Paper No. 99, 30 pp.
\end{thebibliography}
\end{document}